\documentclass[10pt]{article}
 \usepackage{amsmath,amsfonts,amssymb,amsthm}
 \usepackage{hyperref}

 \theoremstyle{plain}
 \newtheorem{theorem}{Theorem}[section]
 \newtheorem{lemma}[theorem]{Lemma}
 \newtheorem{corollary}[theorem]{Corollary}
 \newtheorem{proposition}[theorem]{Proposition}
 
  \newtheorem{question}[theorem]{Question}

 \theoremstyle{definition}
 \newtheorem{definition}[theorem]{Definition}
 \newtheorem{definitionnotation}[theorem]{Notation}
 
 \newtheorem{lemmadef}[theorem]{Definition-Lemma}

 \theoremstyle{remark}
 \newtheorem{remark}[theorem]{Remark}

\newcommand{\spec}{\operatorname{Spec}}

\newcommand{\loc}{\operatorname{Frac}}

\newcommand{\hg}{\operatorname{H}}

\newcommand{\chara}{\operatorname{char}}
\newcommand{\iso}{\operatorname{Iso}}
\newcommand{\sections}{\operatorname{Sections}}
\newcommand{\cls}{\overline{K}}
\newcommand{\projection}{\operatorname{pr}}
\newcommand{\jac}{\operatorname{Jac}}

\newcommand{\homo}{\operatorname{Hom}}

\newcommand{\mni}{\medskip\noindent}

\newcommand{\wt}{\widetilde}

\newcommand{\ol}{\overline}

\newcommand{\SP}{\text{Spec }}

\newcommand{\Cc}{C}

\newcommand{\bB}{B}

\usepackage{graphicx}
\usepackage{xcolor}
\usepackage{float}
\usepackage{amssymb}
\usepackage{amsmath}
\usepackage{mathrsfs}
\usepackage{bm}
\usepackage[all,cmtip]{xy}
\usepackage{enumitem}
\usepackage[utf8]{inputenc}
\usepackage[english]{babel}
\usepackage{multirow}
\usepackage{mathtools}
\usepackage{scalerel,stackengine}
\usepackage{stmaryrd}
\usepackage{datetime}
\usepackage[title]{appendix}
\usepackage{url}
\usepackage{tabulary}
\usepackage{booktabs}

\title{Pseudo-N\'eron Model and Restriction of Sections}

\title{%
  Pseudo-N\'eron Model and Restriction of Sections \\
  \large with an appendix by Jason Michael Starr}

\author{Santai Qu}

\date{\today}

\begin{document}

\maketitle

\begin{abstract}
We introduce the notion of pseudo-N\'eron model and give new examples of varieties admitting pseudo-N\'eron models other than Abelian varieties.  As an application of pseudo-N\'eron models, given a scheme admitting a finite morphism to an Abelian scheme over a positive-dimensional base, we prove that for a very general genus-0, degree-$d$ curve in the base with $d$ sufficiently large, every section of the scheme over the curve is contained in a unique section over the entire base.
\end{abstract}

\tableofcontents

\section{Introduction}

\subsection{Main results}

By a Dedekind scheme, we always mean an irreducible, Noetherian and normal scheme of dimension one.  Let $S$ be a Dedekind scheme with function field $K$.  Let $X_K$ be a smooth and separated $K$-scheme of finite type.  We say that $X$ is an \emph{$S$-model} of $X_K$ if $X$ is an $S$-scheme with generic fiber isomorphic to $X_K$.  A N\'eron model of $X_K$ is an $S$-model satisfying a universal property of extending morphisms.  This extends the smooth variety $X_K$ to a family of smooth varieties over $S$.  The precise definition is the following.

\begin{definition}(\cite{BLR}, Def.1.2/1, p.12)\label{Nerondefn}
Let $X_K$ be  a smooth and separated $K$-scheme of finite type.  A \emph{N\'eron model} of $X_K$ is an $S$-model $X$ which is \emph{smooth}, separated, and of finite type, and which satisfies the following universal property, called the \emph{N\'eron mapping property}:

For each smooth $S$-scheme $Y$ and each $K$-morphism $u_K: Y_K\rightarrow X_K$ there is a unique $S$-morphism $u: Y\rightarrow X$ extending $u_K$.
\end{definition}

From the uniqueness of the morphism extension, it is easy to see that a N\'eron model is unique as soon as it exists.  If $X_K$ is an Abelian variety over $K$, then the existence of N\'eron model is proved in the welcomed survey book \cite{BLR}.  However, the N\'eron model of an Abelian variety is not necessary an Abelian scheme over the Dedekind scheme $S$ (cf. \cite{BLR}, Theorem 1.4/3, p.19).

\begin{theorem}(\cite{BLR}, Theorem 1.4/3, p.19)\label{Nerontheoremblr}
Let $X_K$ be an abelian variety over $K$.  Then $X_K$ admits a N\'eron model $X$ over S.  
\end{theorem}

The N\'eron model is a very important tool in both arithmetic geometry and algebraic geometry.  In the article \cite{GJ}, Tom Graber and Jason Michael Starr introduced the technique of using N\'eron models to prove the theorem of restriction of sections for families of Abelian varieties (Theorem~\ref{maintheoremjason}).  To state their theorems, we cite the following definition from \cite{GJ}.

\begin{definition}(\cite{GJ}, p.312)\label{linesdef}
Let $k$ be an algebraically closed field.  Fix a generically finite, generically unramified morphism $u_0: S\rightarrow \mathbb{P}^n_k$.  We define
\begin{itemize}
\item{an \emph{$u_0$-line} is a curve in $S$ of the form $S\times_{\mathbb{P}^n_k}L$ for a line $L\subset \mathbb{P}^n_k$;}
\item{an \emph{$u_0$-conic} is a curve in $S$ of the form $S\times_{\mathbb{P}^n_k}C$ for a plane conic $C\subset \mathbb{P}^n_k$;}
\item{an \emph{$u_0$-line-pair} is a curve in $S$ of the form $S\times_{\mathbb{P}^n_k}L$, where $L=L_1\cup L_2$ for a pair of incident lines in $\mathbb{P}^n_k$;}
\item{an \emph{$u_0$-smooth-curve} is an irreducible smooth curve in $S$ of the form $S\times_{\mathbb{P}^n_k} C_0$ for a smooth curve $C_0\subset \mathbb{P}^n_k$;}
\item{an \emph{$u_0$-curve-pair} of degree $d+2$ is a curve in $S$ of the form $S\times_{\mathbb{P}^n_k}C$, where $C=C_0\cup C_1$ for a pair of incident curves in $\mathbb{P}^n_k$ where $C_0$ is a genus zero, smooth curve of degree $d$, and $C_1$ is a conic;}
\item{an \emph{$u_0$-planar surface} is a surface in $S$ of the form $S\times_{\mathbb{P}^n_k}\Sigma$ for a $2$-plane $\Sigma\subset \mathbb{P}^n_k$.}
\end{itemize}
\end{definition}

Note that, by Bertini's theorem, for sufficiently general line, conic, and plane, the corresponding $u_0$-line, $u_0$-conic, and $u_0$-planar surface will be smooth.  By abuse of notations, we will just say line, conic, line-pair, curve-pair, planar surface, and smooth curve in $S$ instead of $u_0$-line, $u_0$-conic, $u_0$-line-pair, $u_0$-curve-pair, $u_0$-planar surface, and $u_0$-smooth-curve.  

Let $k$ be an uncountable algebraically closed field.  We say a subset of a scheme is \emph{general}, resp. \emph{very general}, if the subset contains an open dense subset, resp. the intersection of a countable collection of open dense subsets.  We say that a property of points in a scheme holds \emph{at a general point}, resp. \emph{at a very general point}, if the set where the property holds is a general subset, resp. a very general subset.

Now, we state the main theorem in \cite{GJ} as following.

\begin{theorem}(\cite{GJ}, Theorem 1.3, p.312)\label{maintheoremjason}
Let $k$ be an uncountable algebraically closed field.  Let $S$ be an integral, normal, quasi-projective $k$-scheme of dimension $b\ge 2$.  Let $A$ be an Abelian scheme over $S$.  For a very general line-pair $C$ in $S$, the restriction map of sections
\[\sections(A/S)\to\sections(A_C/C)\]
is a bijection.  The theorem also holds with $C$ a very general planar surface in $S$.  If $\chara k=0$, this also holds with $C$ a very general conic in $S$.
\end{theorem}

The main application of N\'eron models in the proof of Theorem~\ref{maintheoremjason} is Lemma 4.13 in \cite{GJ}, p.323.  However, going over the proof, it is easy to see that only the existence of extensions of morphisms is needed, but not the uniqueness, and this is also the case for many other applications of N\'eron models.  This leads us to weaken the definition of N\'eron mapping property, and consider a weak version of N\'eron model.

\begin{definition}\label{weakextensionproperty}
Let $X$ be a separated, flat scheme of finite type over $S$.  We say $X$ has the \emph{weak extension property} if for every smooth scheme $Z$ over $S$ and every $K$-morphism $u_K:Z_K\rightarrow X_K$, there exists an $S$-morphism $u:Z\rightarrow X$ extending $u_K$.
\end{definition}

\begin{definition}\label{weakextensiondefn}
Let $X_K$ be a smooth, finite type and separated $K$-variety.  Suppose that $X$ is a separated, flat and finite type scheme over $S$ with generic fiber $X_K$.  We say that $X$ is a \emph{pseudo-N\'eron model of its generic fiber} if $X$ satisfies the weak extension property.
\end{definition}

We note that, in our definition for pseudo-N\'eron models, we do not require that $X$ is normal or regular since after an \'etale base change $T\to S$, $X_T$ is not necessarily normal or regular.  And, we stress that, unlike N\'eron models, a pseudo-N\'eron model is always not unique.

By \cite{BLR} Proposition 1.2/8, we know that every Abelian scheme over $S$ satisfies the weak extension property.  Moreover, from Theorem~\ref{Nerontheoremblr}, every Abelian variety has a N\'eron model, and hence a pseudo-N\'eron model.  So it is natural to ask the following question.

\begin{question}\label{firstquestion}
Is there any other class of varieties, besides Abelian schemes and Abelian varieties, satisfying the weak extension property or admitting pseudo-N\'eron models?
\end{question}

In the first part of this article, we give a positive answer to this question.  It turns out that the existence of pseudo-N\'eron models is closely related to the non-existence of rational curves on the variety (Corollary~\ref{pseudoliu}).  Our new example of pseudo-N\'eron models is the following, which will be restated and proved as Corollary~\ref{newgeometric}.

\begin{theorem}(New Examples)\label{newexamples}
Let $k$ be an uncountable algebraically closed field.  Let $S$ be a Dedekind scheme of finite type over $k$ with field of functions $K$ (e.g. $S$ is a smooth curve over $k$).  Let $d$ be an integer prime to $\chara k$.  Let $X_K$ be a smooth $K$-variety admitting a finite morphism to a very general smooth hypersurface of degree $d\ge 2n-1$ defined over $k$ in $\mathbb{P}^n_K$.  Then $X_K$ has a pseudo-N\'eron model over $S$.  

In particular, every smooth subvariety of a very general hypersurface of degree $d\ge 2n-1$ defined over $k$ in $\mathbb{P}^n_K$ has a pseudo-N\'eron model.
\end{theorem}

In this second part of this article, we will use pseudo-N\'eron models to restate the Lemma 4.13 in \cite{GJ} in a more general set up.  As a consequence, we prove that there exists a broader class of varieties for which Theorem~\ref{maintheoremjason} holds for higher order curve-pairs and smooth curves.

\begin{theorem}\label{generalizedjasonmaintheorem}
Let $k$ an uncountable algebraically closed field of characteristic zero.  Let $S$ be an integral, normal, quasi-projective $k$-scheme of dimension $b\ge 2$.  Let $X$ be a smooth $S$-scheme admitting a finite morphism $f: X\rightarrow A$ to an Abelian scheme $A$ over $S$.  Let $e$ be the fiber dimension of $\iso(A)$ where $\iso(A)$ is the isotrivial factor of $A$ (see Definition-Lemma~\ref{isotrivialfactor} and Remark~\ref{fiberdimension}).  Let $d$ be a positive even integer.  

Then, for $d>2e$, every section of $X_C$ over a very general genus-0 and degree-$(d+2)$ curve-pair or a very general genus-0, degree-$(d+2)$ smooth curve $C$ is the restriction of a unique global section of $X$ over $S$. 
\end{theorem}

\subsection{Review of theorems of sections}

Now, we give a brief review of theorems of sections in literatures and compare our main theorem with these results.

A complex variety $V$ is said to be \emph{rationally connected} if two general points of $V$ can be joined by a rational curve (\cite{Kollar}, Definition 3.2, p.199).  In \cite{GHS}, it is proved that a one-parameter family of rationally connected complex varieties has a section.  

\begin{theorem}(\cite{GHS}, Theorem 1.1, p.57)\label{GHS}
Let $f: X\to B$ be a proper morphism of complex varieties with $B$ a smooth curve.  If the general fiber of $f$ is rationally connected, then $f$ has a section.
\end{theorem}

\begin{definition}(\cite{GHMS}, Def.1.2, p.672)\label{defGHMS}
Let $\pi: X\to B$ be an arbitrary morphism of complex varieties.  By a \emph{pseudosection} of $\pi$ we will mean a subvariety $Z\subset X$ such that the restriction $\pi|_Z: Z\to B$ is dominant with rationally connected general fiber.
\end{definition}

In \cite{GHMS}, the authors prove the coverse of Theorem~\ref{GHS} as following.

\begin{theorem}(\cite{GHMS}, Theorem 1.3, p.672, \cite{GJ}, Theorem 1.1, p.311)\label{GHMS}
Let $\pi: X\to B$ be a proper morphism of complex varieties.  If $\pi$ admits a section when restricted to a very general sufficiently positive curve in $B$, then there exists a pseudosection of $\pi$.
\end{theorem}

However, the theorem asserts only the existence of a pseudosection of $\pi$; it does not claim any direct connection between the sections of $X_C\to C$ over very general curves $C$ and the pseudosection.  So the following question is asked in \cite{GHMS} and \cite{GJ}.
\begin{question}(\cite{GJ}, Conjecture 1.2, p.311, \cite{GHMS}, Question 7.1, p.689)\label{questionghms}
If $\pi: X\to B$ is a morphism of complex varieties, then for a very general, sufficiently positive curve $C\subset B$, does every section of the restricted family $X_C=\pi^{-1}(C)\to C$ take values in a pseudosection?
\end{question}

On the other hand, in Theorem~\ref{GHMS}, the genus and degree of the very general curve depend on the relative dimension of $\pi$ (see the statement of theorem in \cite{GHMS}, Theorem 1.3, p.672).  The genus and degree can grow enormously fast with respect to the relative dimension of $\pi: X\to B$.  So it is natural to ask the following questions.

\begin{question}(\cite{GHMS}, Section 7.3, p.689)\label{degreequestion}
Can we eliminate the dependence of the family of curves on the relative dimension of $\pi$ in Theorem~\ref{GHMS}?  
\end{question}

The answer of this question is ``no".  The detailed proof can be found in \cite{Jason}.  A sketch of the argument could also be found in \cite{GHMS}, Section 7.3.  Then, a further question is the following.

\begin{question}\label{fastquestion}
If the dependence in Theorem~\ref{GHMS} can not be eliminated, how fast do the genus and degree of the family of curves grow?
\end{question}

One extreme special case of Question~\ref{questionghms} and Question~\ref{degreequestion} is that $X$ is an Abelian scheme over $B$.  In this case, since the fibers contain no rational curves, every pseudosection is a rational section, and every rational section is everywhere defined.  Then, Theorem~\ref{maintheoremjason} gives positive answers to both Question~\ref{questionghms} and Question~\ref{degreequestion} when $X$ is an Abelian scheme over $B$.

When $X$ is a smooth scheme admitting a finite morphism to an Abelian scheme $A$ over $B$, Theorem~\ref{generalizedjasonmaintheorem} gives a positive answer to Question~\ref{questionghms} and Question~\ref{fastquestion}.  The genus of curves is zero as in Theorem~\ref{maintheoremjason}.  And, the degree of the curves grows at a \emph{linear} rate with respect to the relative dimension of the isotrivial factor of the Abelian scheme.

\noindent{\bf Acknowledgement:}  The author is very grateful to his advisor Prof. Jason Michael Starr for introducing this problem, his consistent support during the proof and his writing for the Appendix~\ref{thebigonlemma}.

\section{Pseudo-N\'eron Model}

\subsection{Basic properties}
In this section, we will assume that $S$ is a \emph{Nagata} Dedekind scheme and $K$ is its function field.  Recall that every scheme of finite type over a field is Nagata.

\begin{lemma}\label{affineneron}
Suppose that $S=\spec R$ is an affine Nagata Dedekind scheme.  Let $Y_K$ be a smooth variety and $Y$ be a normal pseudo-N\'eron model over $S$.  Let $X_K$ be a smooth $K$-variety with a finite $K$-morphism $f: X_K \rightarrow Y_K$.  Then there exists a flat normal $S$-scheme $X$ admiting a finite morphism $g: X\rightarrow Y$ which extends $f$.
\end{lemma}

\begin{proof}
First consider the affine case.  Let $\spec A_K$ be an affine open subset of $Y_K$ and $\spec B_K=f^{-1}(\spec A_K)$.  Suppose that $\spec A$ is an open affine in $Y$ with generic fiber $\spec A_K$, where $K=\loc(R)$ and $A_K=A\otimes_RK$.  We claim that there exists a finite $A$-algebra $C$ such that $B_K=C\otimes_R K$.  $\loc(B_K)$ is a finite field extension of $\loc(A_K)$ since $B_K$ is finite over $A_K$.  Let $B'$ be the integral closure of $A$ in $\loc(B_K)$.  We have the ring $A$ is Nagata (\cite{Liuqing}, Prop.8.2.29(b), p.340, and Def.8.2.30, p.341), and hence $B'$ is finite over $A$ (\cite{Liuqing}, Def.8.2.27, p.340).  Now, take $C=B_K\cap B'$.  Then since $A$ is Noetherian we have $C$ is finite over $A$, and by construction $C$ is the integral closure of $A$ in $B_K$.  It is easy to check that $B_K=C\otimes_RK$.

Now, take an affine open covering of $Y_K$, and hence an affine open covering of $X_K$.  By the construction in the affine case, $C$ is uniquely determined by $B_K$, $A_K$ and $A$.  Thus, we can glue the $\spec C$ as above to form the normal scheme $X$ with a canonical finite morphism $g: X\rightarrow Y$, which is of finite type, separated and flat over $S$.
\end{proof}

\begin{lemma}\label{finitemap}
Let $Y$ be a separated, flat $S$-scheme of finite type satisfying the weak extension property.  Suppose that $X$ is an integral $S$-scheme with a finite $S$-morphism $f:X\rightarrow Y$.  Then $X$ satisfies the weak extension property.
\end{lemma}

\begin{remark}
The Dedekind scheme $S$ does not have to be Nagata in this lemma.
\end{remark}

\begin{proof}
\emph{Step 1}: Assume that $S=\spec R$ is an affine Dedekind scheme.  Let $Z$ be an irreducible smooth $S$-scheme with generic fiber $Z_K$ and a $K$-morphism $u_K: Z_K\rightarrow X_K$.  We note that if $\spec A$ is an affine open in $X$ then $A\otimes_RK$ is also an integral domain, so the generic fiber $X_K$ is also an integral scheme with the same function field $K(X_K)=K(X)$ as $X$.

First we assume that $X$ and $Y$ are affine.  Since $Y$ satisfies the weak extension property, $f_K\circ u_K$ extends to an $S$-morphism $g: Z\rightarrow Y$.  Denote $\zeta$ by one of the generic points of codimension one irreducible subsets of $Z$.  Then, since $Z$ is normal, $\mathcal{O}_{Z,\zeta}$ is a discrete valuation ring.  And we have the following commutative diagram
\[\xymatrix{
\spec K(Z)\ar[r]\ar[d]& X\ar[d]^{f} \\
\spec \mathcal{O}_{Z,\zeta}\ar[r]\ar@{-->}[ur]_{u_{\zeta}}& Y\\
}\]
where $\spec K(Z)\rightarrow X$ is induced by the map $K(X_K)\rightarrow K(Z_K)$.  Moreover, by the properness of the morphism $f$, there exists a unique morphism $u_{\zeta}: \spec \mathcal{O}_{Z,\zeta}\rightarrow X$ making the diagram commute.  Since $Z$ is locally of finite type over $S$, the morphism $u_{\zeta}$ can be extended to a neighborhood $V$ of $\zeta$ in $Z$.  We denote this morphism by $u_V: V\rightarrow X$.  Checking every open affine $\spec C$ in $V$, since $Z$ is an integral scheme, we have that the generic fiber of the morphism from $\spec C$ to $X$ is the same as the restriction of $u_K$ because they are giving the same morphism when viewed as restriction of $K(X)\rightarrow K(Z)$.  Moreover, suppose that there are two codimension one points $\zeta_1$ and $\zeta_2$, and they are giving two extensions $V_1\rightarrow X$ and $V_2\rightarrow X$ respectively.  Since $f$ is separated, it is easy to see that these two morphisms agree on every open affine in the overlap $V_1\cap V_2$ since they have the same generic fiber and hence give the same map in $K(X)\rightarrow K(Z)$.  Therefore, the morphism $u_K$ can be extended to a rational map defined over every condimension one point on $Z$.  Hence, by \cite{BLR} Lemma 4.4/2, since $X$ is affine, this rational map is actually defined everywhere.

Now, when $X$ and $Y$ are not affine, consider an open affine covering of $Y$, which induces an affine covering of $X$, and extend $u_K$ for every such open affine of $X$.  By the same reason as above, since any two extensions give the same morphism on generic fibers, these extensions on affines of $X$ can be glued, and hence $X$ satisfies the weak extension property.

\emph{Step 2}: $S$ is a Dedekind scheme, not necessarily affine.  Take an affine covering $\{U_1,\cdots, U_n\}$ of $S$.  Let $X_i$, $Y_i$, $Z_i$ and $(u_K)_i$ be the base changes of $X$, $Y$, $Z$ and $u_K$ from $S$ to $U_i$ respectively.  From step 1, we know that every $(u_K)_i$ can be extended to the whole $Z_i$.  Now, cover each $U_i\cap U_j$ by open affines and check over each these affines.  Again, since these extensions give the same morphism on generic fibers of their overlap, the extensions can be glued and give an extension of $u_K$ to the whole $Z$.  So $X$ satisfies the weak extension property.
\end{proof}

\begin{remark}
Let $\mathcal{C}$ denote the category of normal $S$-schemes with finite morphisms.  Then the above lemma asserts that normal $S$-schemes with the weak extension property form a fully faithful subcategory of $\mathcal{C}$.
\end{remark}

\begin{theorem}\label{neron}
Let $S$ be a Nagata Dedekind scheme with generic point $\spec K$.  Let $X_K$ be a smooth scheme admitting a finite $K$-morphism to a smooth, separated variety $Y_K$ of finite type which has a normal pseudo-N\'eron model $Y$ over $S$.  Then $X_K$ has a normal pseudo-N\'eron model $X$ over $S$.  
\end{theorem}

\begin{proof}
Let $\{U_1,\cdots, U_n\}$ be a finite open affine covering of $S$ and $Y^i$ be the inverse images of each $U_i$ such that they form an open covering of $Y$.  Then $Y_K^i$ and $X_K^i=f^{-1}(Y_K^i)$ form an open covering of $Y_K$ and $X_K$ respectively.  For each $i$, $Y^i$ is flat, separated and of finite type over the affine Nagata Dedekind scheme $U_i$.  Take $X^i$ to be the $U_i$-model of $X_K^i$ as constructed in Lemma~\ref{affineneron} and $g_i: X^i\rightarrow Y^i$ to be the corresponding $U_i$-morphism extending $f|_{Y_K^i}$.

Cover each $Y^i$ by affine opens.  Suppose that $\spec A^i$ and $\spec A^j$ are two such affine opens in $Y^i$ and $Y^j$ respectively with $i\not=j$.  Since $\spec A^i$ and $\spec A^j$ are affine schemes over affine bases, $U_i$ and $U_j$, also their generic fibers are affine.  Let the inverse image of $\spec A^i_K$ (resp. $\spec A^j_K$) in $X_K^i$ (resp. $X^j_K$) be $\spec B^i_K$ (resp. $\spec B^j_K$).  Then, as in the construction in Lemma~\ref{affineneron}, we can construct an affine open $\spec C^i$ (resp. $\spec C^j$) as the integral closure of $A^i$ in $B^i_K$ (resp. $A^j$ in $B^j_K$).  Now, by Nike's trick (\cite{foag}, Prop.5.3.1, p.157), cover $\spec A_i\cap \spec A_j$ by principal open affines.  Then they have affine generic fibers since $\spec A^i$ and $\spec A^j$ have affine bases.  Because the process of taking integral closure is unique up to a unique isomorphism and compatible with localization, the affine opens $\spec C^i$ and $\spec C^j$ with morphisms $g_i$ and $g_j$ can be glued.  By the uniqueness of taking integral closure, we can make the same gluing for other pairs of affine opens in the fixed affine covering of $Y^i$ and $Y^j$.  Thus, we obtain a gluing of $X^i$ and $X^j$.  Similarly, $X^1,\cdots, X^n$ glue to be an $S$-scheme $X$ admitting a finite $S$-morphism to the normal scheme $Y$.

Take the $S$-model $X$ of $X_K$ as above.  By applying Lemma~\ref{finitemap} to the scheme $X$, we have this normal $S$-scheme satisfies the weak extension property.
\end{proof}

This theorem gives us a strategy.  Suppose that $S$ is a Nagata Dedekind scheme.  Then every time we have a class of varieties admitting normal pseudo-N\'eron models, by considering smooth varieties with finite morphisms to the varieties in this class, we will get a new class of varieties admitting normal pseudo-N\'eron models.  As a first result, we know that all varieties admitting finite morphisms to Abelian varieties have normal pseudo-N\'eron models.  In particular, every smooth subvariety of an Abelian variety has a normal pseudo-N\'eron model.  

\subsection{Application of rational curves}

In \cite{Liu}, Qing Liu and Jilong Tong proved theorems about N\'eron models of smooth proper curves of positive genus, see \cite{Liu}, Theorem 1.1, p.7019, for details.  In our situation, their result (\cite{Liu}, Prop.4.13, p.7031) in the higher dimensional case can be used to construct new examples of varieties admitting pseudo-N\'eron models.  We start with the basic notion of rational curves as following.

\begin{definition}(\cite{Liu}, p.7031)
Let $V$ be a variety over an algebraically closed field $k$.  We say that $V$ \emph{contains a rational curve} if there is a locally closed subscheme of $V$ which is isomorphic to an open dense subscheme of $\mathbb{P}^1_k$.  
\end{definition}
If $V$ is proper over $k$, then every morphism from an open dense of $\mathbb{P}^1_k$ can be extended to the whole $\mathbb{P}^1_k$ (\cite{Liuqing}, Cor. 4.1.17, p.119).  Moreover, by L\"uroth's theorem, our definition is the same as the existence of a nonconstant morphism from $\mathbb{P}^1_k$ to $V$ (\cite{Kollar}, Definition 2.6, p.105).

\begin{proposition}(\cite{Liu}, Prop.4.13, p.7031)\label{neronliu}
Let $S$ be a Dedekind scheme with field of functions $K$.  Let $X_K$ be a smooth proper variety over $K$.  Suppose $X_K$ has a proper regular $S$-model $X$ such that no geometric fiber $X_{\overline{s}}$, $s\in S$, contains a rational curve.  Then the smooth locus $X_{sm}$ of $X$ is the N\'eron model of $X_K$.  
\end{proposition}

Note that the regularity of $X$ in the above theorem is only used to apply \cite{Liu} Cor.3.12.  Thus, in the case of pseudo-N\'eron models, the same proof gives the following corollary.  

\begin{corollary}\label{pseudoliu}
Let $S$ be a Dedekind scheme with field of functions $K$.  Let $X_K$ be a smooth proper variety over $K$.  Suppose $X_K$ has a proper and flat $S$-model $X$ such that no geometric fiber $X_{\overline{s}}$, $s\in S$, contains a rational curve.  Then $X$ satisfies the weak extension property, i.e., $X$ is a pseudo-N\'eron model of $X_K$.
\end{corollary}

The following lemma is well-known.
\begin{lemma}\label{etalecurve}
Let $k$ be an algebraically closed field.  Let $f: X\rightarrow Y$ be an \'etale surjective $k$-morphism of proper $k$-varieties.  If $X$ does not contain any rational curve, then $Y$ does not contain any rational curve.
\end{lemma}

\begin{proof}
Suppose there exists a rational curve on $Y$, i.e. a nonconstant $k$-morphism $\mathbb{P}^1_k\rightarrow Y$.  Since $f$ is surjective, base change gives a nonconstant morphism $X\times_Y\mathbb{P}^1_k\rightarrow X$, and \'etale surjective morphism $\overline{f}: X\times_Y\mathbb{P}^1_k\rightarrow \mathbb{P}^1_k$.  Since \'etale morphisms preserve the dimensions of local rings (\cite{Liuqing}, Prop.4.3.23), $X\times_Y\mathbb{P}^1_k$ is a smooth proper curve over $k$, and $\overline{f}$ is quasi-finite.  Let $C$ be one of the connected components of $X\times_Y\mathbb{P}^1_k$, which is a smooth, connected, projective curve (\cite{Hart}, Prop.II.6.7).  Then $\overline{f}: C\rightarrow \mathbb{P}^1_k$ is a finite \'etale morphism of smooth projective curves (\cite{Hart}, Prop.II.6.8).  Therefore, the restriction of $\overline{f}$ on $C$ is an isomorphism (\cite{Liuqing}, Cor.7.4.20), which gives a rational curve on $X$.  This contradicts the hypothesis that $X$ does not contain rational curve.
\end{proof}

Lemma~\ref{etalecurve} gives an immediate application of Corollary~\ref{pseudoliu} as following.

\begin{corollary}\label{etalesurjective}
Let $X$ be a proper pseudo-N\'eron $S$-model of $X_K$ such that no geomertric fiber contains rational curves, as in Corollary~\ref{pseudoliu}.  Suppose that $Y$ is a proper $S$-scheme with smooth generic fiber $Y_K$ and there exists an \'etale surjective morphism $f: X\rightarrow Y$.  Then $Y$ is a pseudo-N\'eron model of $Y_K$.  
\end{corollary}

There are many varieties which do not contain any rational curve.  One of the typical examples is very general hypersurfaces of large degree.  We cite the following result.

\begin{theorem}(\cite{nocurves}, Theorem 1.2)
Let $k$ be an algebraically closed field.  For $d\ge 2n-1$, a very general hypersurface $X\subset \mathbb{P}^n_k$ of degree $d$ contains no rational curves, and moreover, the locus of hypersurfaces that contain rational curves will have codimension at least $d-2n+2$.
\end{theorem}

\begin{lemma}\label{curvelemma}
Let $R$ be a Nagata DVR with fraction field $K$ and residue field $k$.  Suppose that $X$ is a proper scheme over $R$ with nonempty fibers.  If $X_{\overline{k}}$ contains no rational curves, then $X_{\overline{K}}$ also contains no rational curves.
\end{lemma}

\begin{proof}
Suppose that there is a nonconstant $K$-morphism $f_{\cls}: \mathbb{P}^1_{\overline{K}}\rightarrow X_{\overline{K}}$.  By limit arguments, there exists a discrete valuation ring $T$ with fraction field $L$, finite over $K$, and residue field $l$ such that $R\subset T\subset \cls$, $T$ dominates $R$, and $f_{\cls}$ is the base change of a nonconstant $L$-morphism $f_L: \mathbb{P}^1_L\rightarrow X_L$.  Consider the generic point of the special fiber $\mathbb{P}^1_l$ which is of codimension one in $\mathbb{P}^1_T$.  Using the valuative criterion of properness, $f_L$ extends uniquely to a $T$-morphism $f_T: V\rightarrow X_T$ where $V$ is an open dense of $\mathbb{P}^1_T$ containing the generic point of $\mathbb{P}^1_l$ (\cite{Liuqing}, Prop.4.1.16, p.119), and hence an open dense of $\mathbb{P}^1_l$.  

Let $\Gamma$ be the normalization of the schematic closure for the graph of $f_T$.  Then the projection $\mathbb{P}^1_T\times_TX_T\rightarrow \mathbb{P}^1_T$ induces a birational morphism $\pi: \Gamma\rightarrow \mathbb{P}^1_T$.  Since $R$ is Nagata, $T$ is finite over $R$, so $T$ is also Nagata (\cite{Liuqing}, Def.8.2.27 and Prop.8.2.29, p.340).  Thus, all the schemes are Nagata and the normalization morphism is finite.  And, hence, $\pi$ is a proper birational morphism.  Let $E$ be the exceptional locus of $\pi$.  By Abhyankar's lemma (\cite{rational}, Theorem 4.26, p.112), $E$ is ruled over its image.  Let $E'$ be its image.  Then $E$ is birationally equivalent over $E'$ to $W\times_{E'} \mathbb{P}^1_{E'}$.  Note that, since $f_T$ is defined over $V$, $E'$ is finitely many closed points in the closed fiber $\mathbb{P}^1_l$ and $E$ is of codimension one.  Then $W$ is dimension zero over $E'$, hence finitely many closed points.  Thus, $W\times_{E'}\mathbb{P}^1_{E'}$ is a finite disjoint copy of $\mathbb{P}^1_{k_j}$ with each $k_j$ a finite extension of $l$.  Base change to the algebraic closure $\overline{l}$ of $l$, then each $\mathbb{P}^1_{k_j}$ splits to finitely many disjoint copies of $\mathbb{P}^1_{\overline{l}}$.  Now, take one of these copies, say, $C$.  If $C$ is mapped to a single point of $X_{\overline{l}}$, then the image of $C$ in $(\mathbb{P}^1_T\times_TX_T)\times_T{\overline{l}}$ is a single point, contradicting that $\Gamma_{\overline{l}}\rightarrow (\mathbb{P}^1_T\times_TX_T)\times_T{\overline{l}}$ is a finite morphism.  Therefore, $C$ is a rational curve in $X_{\overline{l}}=X_{\overline{k}}$.  And this contradiction shows that $X_K$ does not contain any rational curve.
\end{proof}

\begin{definition}
Let $R$ be a DVR with fraction field $K$ and $E=R^{\times}\cup \{0\}$.  Let $H$ be a hypersurface in $\mathbb{P}^n_K$.  We say that $(H,f)$ is a \emph{unitary hypersurface} if the defining equation $f$ of $H$ has coefficients in $E$.
\end{definition}

\begin{theorem}\label{hypersurface}
Let $R$ be a Nagata DVR with fraction field $K$.  Suppose that the residue field $k$ is uncountable and algebraically closed, and $d$ is an integer prime to $\chara{k}$.  Then, there exists unitary hypersurfaces of degree $d\ge 2n-1$ in $\mathbb{P}^n_K$ admitting a N\'eron model.
\end{theorem}

\begin{proof}
Suppose that $X_K=V_+(f)_K\subset \mathbb{P}^n_K$ is a unitary hypersurface defined by an irreducible homogeneous polynomial $f$ of degree $d$ in $n+1$ variables.  Since all the nonzero coefficients of $f$ are in the group of units $E$ of $K$, there is no term in $f$ vanishing in the residue field of $R$, so the specialization $X_k=V_+(f)_k$ is a hypersurface of degree $d$ in $\mathbb{P}^n_k$.  Conversely, every hypersurface of degree $d$ in $\mathbb{P}^n_k$ arises as a specialization of some unitary hypersurface of degree $d$ in $\mathbb{P}^n_K$.  

Set $N={n+d\choose d}-1$.  Let $E$ be the space of unitary hypersurfaces in $\mathbb{P}^n_K$.  The argument above gives a surjective map of parameter spaces $F: E\rightarrow \mathbb{P}^N_k$ by sending $(X_K,f)$ to its specialization $X_k$.  Let $U$ be an intersection of countably many open dense subsets of $\mathbb{P}^N_k$ such that every member in $U$ is smooth without rational curves.  Take $X_K\in F^{-1}(U)$ a $K$-point.  By Lemma~\ref{curvelemma}, there is no rational curve on $X_K$.  Let $X=V_+(f)\subset \mathbb{P}^n_R$ be the $R$-model of $X_K$.  

Since $f$ is irreducible, $X$ is an integral hypersurface.  Thus, $X$ is flat over $\spec R$ and every nonempty fiber is irreducible of dimension $n-1$(\cite{Liuqing}, Cor.4.3.10, p.137).  Let $\text{Fitt}_{n-1}(\Omega^1_{X/R})$ be the $(n-1)$-th Fitting ideal of $\Omega^1_{X/R}$, which is a coherent ideal sheaf of $\mathcal{O}_X$.  Then, $\text{Fitt}_{n-1}(\Omega^1_{X_k/k})$ is equal to $(\text{Fitt}_{n-1}(\Omega^1_{X/R}))\cdot\mathcal{O}_{X_k}$ (\cite{eisenbud}, Cor.20.5, p.498).  Since $X_k$ is smooth, $\Omega^1_{X_k/k}$ is locally free of rank $n-1$.  Thus, $\text{Fitt}_{n-1}(\Omega^1_{X_k/k})$ is equal to $\mathcal{O}_{X_k}$ (\cite{eisenbud}, Prop.20.6, p.498).  And hence, $\text{Fitt}_{n-1}(\Omega^1_{X/R})$ is equal to $\mathcal{O}_X$.  Then, $\Omega^1_{X/R}$ can be locally generated by $n-1$ elements (\cite{eisenbud}, Prop.20.6, p.498).  So $\Omega^1_{X_K/K}$ can be locally generated by $n-1$ elements, and hence locally free of rank $n-1$ (\cite{Liuqing}, Lemma 6.2.1, p.220, and \cite{BLR}, Prop.2.2/15, p.43).  Thus $X_K$ is smooth.  At this stage, every fiber of $X$ is smooth and $X$ is flat over $R$, then $X$ is a smooth $R$-scheme (\cite{BLR}, Prop.2.4/8, p.53).  Therefore, by Lemma~\ref{curvelemma} and Proposition~\ref{neronliu}, $X$ is the N\'eron model of $X_K$.  
\end{proof}

This theorem gives a direct corollary in the geometric setting as following.

\begin{corollary}\label{geometricthm}
Let $k$ be an uncountable algebraically closed field.  Let $S$ be a Dedekind scheme of finite type over $k$ with field of functions $K$ (for example, $S$ is a smooth curve over $k$).  Let $d$ be an integer prime to $\chara k$.  Then, a very general smooth hypersurface of degree $d\ge 2n-1$ defined over $k$ in $\mathbb{P}^n_K$ has a N\'eron model.  In particular, the N\'eron model of such a hypersurface is the constant family over $S$.
\end{corollary}

Note that we say a $K$-scheme $X$ is defined over $k$ if there exists a $k$-scheme $Y$ such that $X$ is isomorphic to $Y\times_{\spec k}\spec K$ (see \cite{Kollar}, Definition 1.15, p.19).

\begin{proof}
$S$ is a Nagata Dedekind scheme (\cite{Liuqing}, Prop.8.2.29, p.340).  First assume that $S=\spec R$ affine.  Let $f$ be a homogeneous polynomial of degree $d\ge 2n-1$ defined over $k$ in $\mathbb{P}^n_k$ such that there is no rational curve on the smooth hypersurface $V_+(f)$.  Define $X_K=V_+(f)_K$ in $\mathbb{P}^n_K$.  Denote $V_+(f)_R\subset \mathbb{P}^n_R$ by $X$, an $R$-model of $X_K$.  Then, exactly the same argument as in Theorem~\ref{hypersurface} shows that $X$ is the N\'eron model of $X_K$.  

Now, take a finite affine covering $\{\spec R_i\}_{i\in I}$ of $S$.  Then there exists a N\'eron model $X^i$ of $X_K$ over $\spec R_i$ for every $i\in I$.  By the uniqueness of N\'eron model and that N\'eron model is local on the base (\cite{BLR}, Prop.1.2/3, p.13), $\{X^i\}_{i\in I}$ glues to be a N\'eron model of $X_K$ over $S$.
\end{proof}

Combining Theorem~\ref{affineneron} and Corollary~\ref{geometricthm}, we get the following corollary.

\begin{corollary}\label{newgeometric}
Keep the notations of Corollary~\ref{geometricthm}.  Let $X_K$ be a smooth $K$-variety admitting a finite morphism to a very general smooth hypersurface of degree $d\ge 2n-1$ defined over $k$ in $\mathbb{P}^n_K$.  Then $X_K$ has a normal pseudo-N\'eron model over $S$.  In particular, every smooth subvariety of a very general hypersurface of degree $d\ge 2n-1$ in $\mathbb{P}^n_K$, where the hypersurface is defined over $k$, has a normal pseudo-N\'eron model.
\end{corollary}

From this Corollary, we see that there are many smooth varieties admitting normal pseudo-N\'eron model over a smooth curve defined over an uncountable algebraically closed field.  In the situation of our corollary, we cannot control the regularity of other fibers except $X_K$.  It is a normal model of $X_K$, but in general not a N\'eron model in the sense of Definition~\ref{Nerondefn}.  Moreover, the variety $X_K$ is not necessarily defined over $k$, unlike the constant case in Corollary~\ref{geometricthm}.

\subsection{Base change properties}

The next lemma shows that pseudo-N\'eron models commute with \'etale extension of the base scheme, which is the analogue of \cite{BLR} Prop.1.2/2 (c) for N\'eron models.

\begin{lemma}\label{etalebasechange}
Let $S$ be a Dedekind scheme with function field $K$ and $X_K$ be a smooth $K$-variety with pseudo-N\'eron model $X$ over $S$.  Suppose that $S'$ is another Dedekind scheme with $S'\rightarrow S$ \'etale and the function field of $S'$ is $K'$.  Let $X_{S'}=X\times_S S'$ and $X_{K'}=X_K\times_K K'$ be its generic fiber.  Then $X_{S'}$ is a pseudo-N\'eron model of $X_{K'}$.
\end{lemma}

\begin{proof}
Let $Z$ be smooth of finite type over $S'$.  Take a $K'$-morphism $Z_{K'}\rightarrow X_{K'}$.  Then, $Z$ is smooth over $S$ and $Z_{K'}$, as the $K$-generic fiber, is smooth over $K$.  By the weak extension property of $X$, there exists an $S$-morphism from $Z\rightarrow X$ extending $Z_{K'}\rightarrow X_K$.  Hence the universal property of fiber products gives $Z\rightarrow X'$ as an extension of $Z_{K'}\rightarrow X_{K'}$.  
\end{proof}

The following lemma is an analogue of \cite{BLR} Prop.1.2/4.  However, since a pseudo-N\'eron model is not unique, we can not have the converse direction as in \cite{BLR} Prop.1.2/4.

\begin{lemma}\label{localbasechange}
Let $S$ be a Dedekind scheme with function field $K$, $X$ finite type over $S$ and it is a pseudo-N\'eron model of its generic fiber.  Then, for each closed point $s\in S$, the $\mathcal{O}_{S,s}$-scheme $X_s=X\times_S\mathcal{O}_{S,s}$ is a pseudo-N\'eron model of its generic fiber.
\end{lemma}

\begin{proof}
Let $Y_s$ be an smooth $\mathcal{O}_{S,s}$-scheme with a $K$-morphism $u_K: Y_{s,K}\rightarrow X_{s,K}$.  By limit arguments (\cite{BLR}, Lemma 1.2/5), there exists a connected open neighborhood $S'$ of $s$, and a smooth $S'$-scheme $Y'$ such that $Y'\times_{S'}\spec\mathcal{O}_{S,s}=Y_s$.  Lemma~\ref{etalebasechange} gives that $X_{S'}=X\times_S S'$ is a pseudo-N\'eron model of $X_K$ over $S'$.  Then, by the weak extension property of $X_{S'}$, $u_K$ extends to an $S'$-morphism $u': Y'\rightarrow X_{S'}$.  Therefore, the base change $u=u'\times_{S'}\spec\mathcal{O}_{S,s}$ is a required extension of $u_K$.
\end{proof}

\begin{definition}\label{universalmodel}
Let $S$ be a Dedekind scheme and let $X$ be an $S$-scheme satisfying the weak extension property.  We say that \emph{$X$ satisfies the weak extension property universally} if for any $S'$ a Dedekind scheme and for any $S'\rightarrow S$ of finite type, the base change $X\times_SS'$ also satisfies the weak extension property.  
\end{definition}

\begin{definition}
Let $S$ be a Dedekind scheme with fraction field $K$.  Let $X_K$ be a smooth, separated $K$-scheme of finite type, and let $X$ be a pseudo-N\'eron model of $X_K$.  We say that $X$ is a \emph{universal pseudo-N\'eron model} of $X_K$ if $X$ satisfies the weak extension property universally.
\end{definition}

\begin{lemma}
Keep the notations and hypothesis as in Corollary~\ref{pseudoliu}.  Then, $X$ is a universal pseudo-N\'eron model of $X_K$.
\end{lemma}

\begin{proof}
Let $g: S'\rightarrow S$ be a morphism of Dedekind schemes of finite type and $X'=X\times_S S'$.  Take $s\in S$ and $t\in S'$ closed points such that $s=g(t)$.  Then the residue field $\kappa(t)$ is finite over $\kappa(s)$, thus $X'_{\overline{t}}=X_{\overline{s}}$, and hence $X'_{\overline{t}}$ does not contain rational curves.  And by Lemma~\ref{curvelemma}, there is no rational curve on the geometric generic fiber of $X'$.  Then $X'$ satisfies the weak extension property by Corollary~\ref{pseudoliu}.
\end{proof}

\begin{lemma}\label{nocurveongeneric}
Let $S$ be a Dedekind scheme with function field $K$.  Let $X_K$ be a smooth and separated variety of finite type over $K$.  If $X_K$ has a proper $S$-model satisfying the weak extension property universally, then $X_K$ contains no rational curve.
\end{lemma}

\begin{proof}
Suppose that $X$ is a proper $S$-model of $X_K$, i.e., $X$ is flat, separated and finite type over $S$ satisfying the weak extension property and has generic fiber $X_K$.  By Lemma~\ref{localbasechange}, we can replace $S$ by $\mathcal{O}_{S,s}$ for any closed point of $S$, and assume that $S=\spec R$ for some discrete valuation ring R.  If $X_K$ contains a rational curve, then there exists a nonconstant $\overline{K}$-morphism $f_{\overline{K}}: \mathbb{P}^1_{\overline{K}}\rightarrow X_{\overline{K}}$.  By a limit argument as in Lemma~\ref{curvelemma}, there exists a DVR in $\overline{K}$ dominating $R$ with fraction field $L$ and residue field $l$ such that $f_{\overline{K}}$ is a base change of a nonconstant morphism $f_L: \mathbb{P}^1_L\rightarrow X_L$.  Since $X$ is a universal pseudo-N\'eron model, $X_T$ also satisfies the weak extension property.

Let $C$ be the normalization of the schematic closure of $f_L$.  Then, $C$ is a proper normal curve over the field $L$ since $X_L$ is Nagata, and hence $i: C\rightarrow X_L$ is a finite morphism.  Moreover, $f_L$ has a unique factorization via $C$, i.e., $f_L=i\circ g_L$ where $g_L: \mathbb{P}^1_L\rightarrow C$ is a morphism of nonsingular proper curves.  Then $g_L$ is finite (\cite{Liuqing}, Lemma 7.3.10 and Cor.4.4.7).  Therefore, $f_L$ is a finite morphism.  

Since $X_T$ satisfies the weak extension property, $f_L$ extends to a $T$-morphism $f_T: \mathbb{P}^1_T\rightarrow X_T$.  Let $f_l$ be the closed fiber of $f_T$.  Then, the morphism $f_l$ is nonconstant.  The same argument as for $f_L$ shows that $f_l$ is finite.  Therefore, $f_T$ is finite (\cite{Liuqing}, Cor.4.4.7).  Then, from Lemma~\ref{finitemap}, $\mathbb{P}^1_T$ satisfies the weak extension property.  Now, consider an $L$-isomorphism $\sigma_L: \mathbb{P}^1_L\rightarrow \mathbb{P}^1_L$.  However, not all these isomorphisms can be extended to be a $T$-morphism $\mathbb{P}^1_T\rightarrow \mathbb{P}^1_T$ (\cite{BLR}, Example 5, p.75), contradicting that $\mathbb{P}^1_T$ satisfies the weak extension property.  Thus, $X_K$ contains no rational curve.
\end{proof}

\begin{remark}
Let $X_K$ be a smooth, separated $K$-scheme of finite type, and let $X$ be a proper $S$-model of $X_K$.  The above lemma and theorem give us the following picture: 
\begin{enumerate}[label=(\roman*)]
\item no rational curve in any geometric fiber,
\item universal pseudo-N\'eron model,
\item no rational curve in the generic geometric fiber.
\end{enumerate}
Then, (i)$\Rightarrow$(ii)$\Rightarrow$(iii).
\end{remark}

\section{Theorem of Restriction of Sections}

\subsection{Higher dimensional pseudo-N\'eron model}

We will need the notion of higher dimensional pseudo-N\'eron model which generalizes definition 4.10 in \cite{GJ}.  

\begin{definition}\label{defweak}(\cite{GJ} Definition 4.10)
Let S be an integral, regular, separated, Noetherian scheme of dimension $b\ge 1$.  A \emph{flat}, finite type, separated morphism $X\rightarrow S$ has the \emph{weak extension property} if for every triple ($Z\rightarrow S, U, s_U$) of
\begin{enumerate}[label=(\roman*)]
\item a smooth morphism $Z\rightarrow S$,
\item a dense, open subset $U\subset S$,
\item and an $S$-morphism $s_U:Z\times_SU\rightarrow X_U$,
\end{enumerate}
there exists a pair $(V,s_V)$ of 
\begin{enumerate}[label=(\roman*)]
\item an open subset $V\subset S$ containing U and all codimension 1 points of S,
\item and an S-morphism $s_V: Z\times_SV\rightarrow X$ whose restriction to $Z\times_SU$ is equal to $s_V$.
\end{enumerate}
\end{definition}

\begin{definition}\label{defmodel}
Let $S$ be an integral, regular, separated, Noetherian scheme of dimension $b\ge 1$.  Let $K$ be the fraction field of $S$, and $X_K$ be a smooth, separated $K$-scheme of finite type.  A flat, finite type, separated $S$-scheme $X$ is called a \emph{pseudo-N\'eron model} of $X_K$ if $X_K$ is isomorphic to its generic fiber and $X$ satisfies the weak extension property as in Definition~\ref{defweak}.
\end{definition}

By a limit argument, it is easy to see that Definition~\ref{defweak} (resp. Definition~\ref{defmodel}) implies Definition~\ref{weakextensionproperty} (resp. Definition~\ref{weakextensiondefn}) when $S$ is a Dedekind scheme, and they agree when $S=\spec R$, where $R$ is a DVR.  Now, we prove the corresponding results for Lemma~\ref{affineneron}, Lemma~\ref{finitemap} and Theorem~\ref{neron}.

\begin{lemma}\label{affineneronhigh}
Suppose that $S$ is an integral, regular, separated, Noetherian Nagata scheme of dimension $b\ge 1$ with fraction field $K$.  Let $Y_K$ be a smooth $K$-variety and $Y$ be its normal pseudo-N\'eron model over $S$.  Let $X_K$ be a smooth $K$-variety with a finite $K$-morphism $f: X_K \rightarrow Y_K$.  Then there exists a flat normal $S$-scheme $X$ admiting a finite morphism $g: X\rightarrow Y$ which extends $f$.
\end{lemma}

\begin{proof}
Since $S$ is Noetherian, we can cover $S$ by finitely many affine opens.  As we assume that $S$ is Nagata, the same proof of Lemma~\ref{affineneron} and Theorem~\ref{neron} gives the extension as claimed.
\end{proof}

\begin{lemma}\label{finitemaphigh}
Keep the same hypothesis of $S$ as in Lemma~\ref{affineneronhigh}.  Let $Y$ be a separated, flat $S$-scheme of finite type satisfying the weak extension property.  Suppose that $X$ is an integral $S$-scheme with a finite $S$-morphism $f:X\rightarrow Y$.  Then $X$ satisfies the weak extension property.
\end{lemma}

\begin{proof}
Let $U$ be a dense open of $S$, and let $Z$ be a smooth $S$-scheme with a $U$-morphism $t_U: Z_U \rightarrow X_U$.  Composing this morphism with $f_U$ gives a $U$-morphism $s_U: Z_U\rightarrow Y_U$.  Since $Y$ satisfies the weak extension property, there exists an open dense $V$ in $S$ containing all the codimension one points, and an extension $s: Z_V\rightarrow Y_V$ of $s_U$.  Up to replacing $S$ by $V$, we can assume that $V$ is the whole $S$.  Cover $S$ and $Y$ by open affines as in Lemma~\ref{finitemap}, then the same proof as in Lemma~\ref{finitemap} completes the proof.
\end{proof}

Therefore, combining the above two lemmata and the proof of Theorem~\ref{neron}, we get the following theorem.

\begin{theorem}\label{neronhigh}
Keep the same hypothesis of $S$ as in Lemma~\ref{affineneronhigh}.  Let $X_K$ be a smooth scheme admitting finite $K$-morphism to a smooth, separated variety $Y_K$ of finite type which has a normal pseudo-N\'eron model $Y$ over $S$.  Then $X_K$ has a normal pseudo-N\'eron model $X$ over $S$.  

The theorem also holds if $X_K$ and $Y_K$ are replaced by some $X_U$ and $Y_U$ defined over a dense open $U$ of $S$.
\end{theorem}

Moreover, Corollary 4.12 in \cite{GJ} and Theorem~\ref{neronhigh} give us the following corollary.

\begin{corollary}\label{higherdimensionneron}
Let $W$ be an integral, regular, separated, Nagata Noetherian scheme of dimension $b\ge 1$.  Let $S$ be a dense open subset of $W$, and let $X$ be a scheme admitting a finite morphism to an Abelian scheme over $S$.  There exists an open subset $\widetilde{S}$ of $W$ containing $S$ and all codimension one points, and there exists a normal pseudo-N\'eron model $\widetilde{X}$ over $\widetilde{S}$ whose restriction over $S$ equals $X$.
\end{corollary}

\begin{remark}
Even though we have these parallel results as in the case when $S$ is a Nagata Dedekind scheme, Definition~\ref{defweak} and Definition~\ref{weakextensionproperty} do not agree in general when $S$ is a Dedekind scheme because we lack the uniqueness of the morphism extensions.
\end{remark}

\subsection{Bertini's theorems for higher order curves}

Recall that a scheme $X$ is called \emph{algebraically simply connected} if for every connected scheme $Y$, and every surjective finite \'etale morphism $f:Y\to X$, the morphism $f$ is an isomorphism (\cite{Debarre}, p.97).  In this section, by a scheme over a field $k$, we mean a scheme that is of finite type over $k$.

\begin{theorem}(\cite{Christian}, Prop.3.1)\label{Bertinigeneral}
Let $X$ be a smooth, algebraically simply connected variety over a field $k$.  Let $N$ be a normal, connected and quasi-projective $k$-scheme.  Let $h: N\rightarrow X$ be a projective $k$-morphism.  If the closed subscheme $N_h$ of $N$ where $h$ is not smooth has codimension at least 2, then the geometric generic fiber of $h$ is connected.
\end{theorem}

\begin{proposition}\label{variantmino}
Let $k$ be a field.  Let $X$ be a smooth, irreducible $k$-scheme that is algebraically simply connected.  Let $Y$ be an irreducible quasi-projective $k$-scheme.  Let $M$ be a normal, irreducible, quasi-projective $k$-scheme.  Let $(h,g): M \to X\times_k Y$ be a $k$-morphism such that $h$ is projective and surjective and such that $g$ is dominant with irreducible geometric generic fiber.  Let $Z$ be an irreducible $k$-scheme.  Let $f: Z \to Y$ be a finite, surjective $k$-morphism.  Denote by $\nu: N\to Z \times_Y M$ the normalization of the fiber product $Z \times_Y M$.  Denote by $h': N \to X$ the composition of $h$ and projection from $N$ to $M$.  If the closed subscheme of $N$ where $h'$ is not smooth has codimension at least 2, then the geometric generic fiber of $h'$ is connected.

\[\xymatrix{
N\ar[r]^{\nu\,\,\,\,\,\,\,\,}\ar[rd]\ar@/_2pc/[ddrr]_{h'} & Z\times_Y M\ar[rr]\ar[d] &  &  Z\ar[d]^f  \\
  &     M\ar[r]^{(h,g)\,\,\,\,\,\,\,\,}\ar[rd]_h &  X\times_k Y\ar[r]\ar[d]^{\projection_1} &  Y  \\
  &                              &  X
}\]

\end{proposition}

\begin{proof}
Since the geometric generic fiber of $g$ is connected, and since $Z$ is irreducible, also the normalization $N$ is irreducible.  Since $Z\times_Y M$ is a $k$-scheme, it is Nagata, and hence the normalization $\nu$ is a finite morphism.  Moreover, since $M$ is quasi-projective, it admits an ample invertible sheaf.  Thus, the finite morphism $N\to M$ is projective.  Therefore, $h'$ is projective.  Thus, the statement reduces to Theorem~\ref{Bertinigeneral}.
\end{proof}

\begin{definition}\label{ramifieldlocusdef}
Let $Y$ be a regular locally Noetherian scheme.  Let $f: Z \to Y$ be a finite surjective morphism that is generically \'etale.  The closed subscheme $R$ inside $Z$ where $f$ is not \'etale is called  \emph{the ramification locus} of $f$.  Let $B$ denote the image of $R$ in $Y$.  The closed subscheme $B$ of $Y$ is called \emph{the branch locus} of $f$.  
\end{definition}

Note that, by Zariski's purity theorem (\cite{Liuqing}, Exercise 8.2.15(c), p.347), the branch locus $B$ of $f$ is either empty or pure of codimension one if $f$ generically separable.  In particular, let $k$ be a field of characteristic zero, and let $Y$ be a smooth $k$-scheme.  Then, if $f: Z\to Y$ is a finite surjective morphism of $k$-schemes that is generically \'etale, the branch locus $B$ of $f$ is either empty or pure of codimension one.  The following lemma about ramification in codimension one is well-known.

\begin{lemma}\label{ramificationcodimensionone}
Let $k$ be a field of characteristic zero.  Let $Z$ be a normal $k$-scheme.  Let $Y$ be a smooth $k$-scheme.  Let $f: Z\to Y$ be a finite surjective morphism that is generically \'etale.  Denote by $R$ the ramification locus of $f$.  Let $D$ be an irreducible component of $R$, and let $E$ be the image of $D$ in $Y$.  Denote by $\eta_D$ (resp. $\eta_E$) the generic point of $D$ (resp. $E$).

Then, there exists a uniformizer $s$ (resp. $t$) for $\widehat{\mathcal{O}}_{Y,\eta_E}$ (resp. $\widehat{\mathcal{O}}_{Z,\eta_D}$) such that $\widehat{\mathcal{O}}_{Y,\eta_E}\to\widehat{\mathcal{O}}_{Z,\eta_D}$ maps $s$ to  $t^e$ for some positive integer $e$.
\end{lemma}

\begin{definition}\label{ramificationindex}
Keep the notations in Lemma~\ref{ramificationcodimensionone}.  The positive integer $e$ obtained in Lemma~\ref{ramificationcodimensionone} is called \emph{the ramification index of $f$ at $\eta_D$}.
\end{definition}

\begin{theorem}\label{bertinismoothness}
Let $k$ be an algebraically closed field of characteristic zero.  Let $Z$ be a normal $k$-scheme.  Let $f: Z\to \mathbb{P}^n_k$ be a finite surjective morphism that is generically \'etale.  Then, for a general smooth curve $C\subset \mathbb{P}^n_k$, the inverse image $f^{-1}(C)$ is a smooth curve.
\end{theorem}

\begin{proof}
Since $Z$ is normal, the singular locus $Z_{sing}$ has codimension at least two (\cite{Liuqing}, Prop.4.2.24, p.131).  Also the closed subset $f(Z_{sing})$ in $\mathbb{P}_k^n$ has codimension at least two since $f$ is finite.  By Kleiman-Bertini's Theorem (\cite{Hart}, Theorem III.10.8, p.273), for a general smooth curve $C$ in the complement of $f(Z_{sing})$, the inverse image $f^{-1}(C)$ is smooth.
\end{proof}

\begin{theorem}\label{bertinihigherordercurves}
Let $k$ be an algebraically closed field of characteristic zero.  Let $f: Z\to \mathbb{P}_k^n$ be a finite surjective morphism from a normal irreducible variety.  Then, for a general genus-0, degree-$d$ curve $C$ in $\mathbb{P}_k^n$, the inverse image $f^{-1}(C)$ is connected. 
\end{theorem}

\begin{proof}
Keep the notations in Proposition~\ref{variantmino}.  Let $Y$ be projective space $\mathbb{P}^n_k$.  Note that $f$ is generically \'etale by the generic smoothness theorem (\cite{foag}, Theorem 25.3.1, p.681).  Let $X$ be the non-stacky locus inside the stack of genus-$0$, degree-$d$ stable maps to $\mathbb{P}^n_k$.  In other words, $X$ is the maximal open subscheme of this stack.  The open subscheme $X$ is algebraically simply connected.  Let $M$ be the universal family of curves over $X$, and let $g$ be the universal morphism.  Then, the generic geometric fiber of $g$ is connected.

To complete the proof, we need to prove that the singular locus of $h'$ inside $N$ has codimension at least $2$.  The codimension one subset of $X$ parameterizes degree-$d$, genus-$0$ curves in $\mathbb{P}^n_k$ that are not transversal to the branch locus.  Thus, away from codimension one points in $X$, the fibers are everywhere smooth (Theorem~\ref{bertinismoothness}).  Moreover, for a genus-0, degree-$d$ curve that is not transversal to the branch locus, the singularities of the fiber of $h'$ occur only over the intersection points of the curve with the branch locus, and this is codimension one in the fiber.  Thus, the total codimension of singular locus of $h'$ in $N$ is at least two.  By Proposition~\ref{variantmino}, for a general genus-0, degree-$d$ curve $C$ in $\mathbb{P}_k^n$, the inverse image $f^{-1}(C)$ is connected. 
\end{proof}

\begin{corollary}\label{bertiniirreducible}
Let $k$ be an algebraically closed field of characteristic zero.  Let $f: Z\to \mathbb{P}_k^n$ be a finite surjective morphism from a normal irreducible variety.  Then, for a general genus-0, degree-$d$ curve $C$ in $\mathbb{P}_k^n$, the inverse image $f^{-1}(C)$ is smooth and irreducible. 
\end{corollary}

\begin{proof}
By Theorem~\ref{bertinismoothness} and Theorem~\ref{bertinihigherordercurves}, for a general genus-0, degree-$d$ curve $C$, $f^{-1}(C)$ is both smooth and connected.  Thus, the inverse image $f^{-1}(C)$ is irreducible.
\end{proof}

\begin{corollary}\label{bertinisections}
Let $k$ be an algebraically closed field of characteristic zero.  Let $f: Z\to \mathbb{P}_k^n$ be a finite surjective morphism from a normal irreducible variety.  Then, for a general genus-0, degree-$d$ curve $C$ in $\mathbb{P}_k^n$, the restriction map of sections
\[\sections(Z/\mathbb{P}^n_k)\to \sections(Z_C/C)\]
is bijective.
\end{corollary}

\begin{proof}
Keep the notations in Lemma~\ref{ramificationcodimensionone}, and let $Y$ be the projective space $\mathbb{P}^n_k$.  By choosing the curve $C$ generally, we can assume that $C$ intersects with every irreducible component of $R$ transversally.  Suppose that $R_1$, $\cdots$, $R_r$ (resp. $B_1$, $\cdots$, $B_r$) are irreducible components of $R$ (resp. $B$).  Let $\widetilde{R}$ (resp. $\widetilde{B}$) be the union of pairwise intersections of the irreducible components of $R$ (resp. $B$), i.e.,
$$\widetilde{R}=\bigcup_{1\le i<j\le r}R_i\cap R_j,$$
$$\widetilde{B}=\bigcup_{1\le i<j\le r}B_i\cap B_j.$$
Since $k$ is algebraically closed, the sigular locus $\widetilde{Z}$ of $Z$ has codimension at least two (\cite{Liuqing}, Prop.4.2.24, p.131).  Let $\widetilde{Y}$ be the image of $\widetilde{Z}$.  Then, by choosing $C$ generally, we can assume further that $C$ does not intersect with $\widetilde{B}\cup\widetilde{Y}$ and that $f^{-1}(C)$ does not intersect with $\widetilde{R}\cup\widetilde{Z}$.  Thus, we may assume that both $R$ and $B$ are irreducible and that $Z$ is smooth.  Let $\eta_B$ (resp. $\eta_R$) be the generic point of $B$ (resp. $R$).  Let $e$ be the ramification index of $f$ at $\eta_{R}$.  By choosing the genus-0, degree-$d$ curve $C$ generally, we assume that $f^{-1}(C)$ is smooth and irreducible (Corollary~\ref{bertiniirreducible}).

\emph{Case 1}: $e>1$.  Suppose that there exists a section $\sigma$ of $Z_C$ over $C$.  Then, $\sigma$ is a closed immersion of smooth irreducible curves.  Thus, $\sigma$ must be an isomorphism (\cite{Hart}, Prop.II.6.8, p.137).  The maximal ideal of the regular local ring $\mathcal{O}_{f^{-1}(C),q}$ is generated by $t$, and the maximal ideal of $\mathcal{O}_{C,p}$ is generated by $s$ such that $s=t^e$.  Thus, the degree of the effective divisor $f^*(p)$ on $f^{-1}(C)$ is at least $e$ (\cite{Hart}, Definition on p.137).  However, since $\text{deg}f=1$, this contradicts that
\[\text{deg}f^*(p)=\text{deg}f\cdot\text{deg}p\]
(\cite{Hart}, Prop.II.6.9, p.138).  Thus, there is no such section $\sigma$, let alone a section of $Z$ over $\mathbb{P}^n_k$.  Then the restriction map of sections is bijective since both sets of sections are empty.

\emph{Case 2}: $e=1$.  Then $f$ is everywhere \'etale.  Since $\mathbb{P}^n_k$ is algebraically simply connected, the finite \'etale cover $f$ is an isomorphism.  Thus, the restriction map of sections is bijective.
\end{proof}

\begin{theorem}\label{bertinisectionsrevised}
Let $k$ be an algebraically closed field of characteristic zero.  Let $Z$ be a normal $k$-scheme that is not necessarily connected.  Let $f: Z\to S$ be a finite surjective morphism to a smooth, connected, quasi-projective $k$-scheme $S$ where $S$ admits a finite, generically \'etale morphism to an open dense subset of $\mathbb{P}^n_k$, $u_0: S\to \mathbb{P}^n_k$.  Then, for a general genus-0, degree-$d$ curve $C$ in $\mathbb{P}_k^n$, the restriction map of sections
\[\sections(Z/S)\to \sections(Z_C/C)\]
is bijective.  
\end{theorem}

\begin{proof}
Let $U$ be the image of $S$ in $\mathbb{P}^n_k$.  Then, it suffices to show that the restriction map of sections
\[\sections(Z/U)\to \sections(Z_C/C)\]
is bijective, where $Z$ is finite, surjective and generically \'etale over $U$.    

First, we assume that $Z$ is connected.  Then, by the same proof of Corollary~\ref{bertiniirreducible}, for a general genus-0, degree-$d$ curve $C$, the inverse image $f^{-1}(C)$ is an irreducible and smooth curve.  If there exists a codimension one point of $Z$ that has ramification index strictly greater than one, then there is no section for $f$ and for the restriction of $f$ on $f^{-1}(C)$.  Thus, we can assume that $f$ is finite, surjective and everywhere \'etale.  If $\text{deg}(f)$ is strictly greater than one, then there is no section for $f$ (\cite{finiteetale}, Prop.5.3.1, p.165).  Since $f$ is flat, $\text{deg}(f|_{f^{-1}(C)})$ is equal to $\text{deg}(f)$ (\cite{Liuqing}, Exercise 5.1.25(a), p.176).  Thus, there is no section for $\text{deg}(f|_{f^{-1}(C)})$ since $f^{-1}(C)$ is irreducible.  If $\text{deg}(f)$ is one, then the restriction map of sections is trivially bijective.

Now, suppose that $Z$ is not connected.  Let $Z_1,\cdots, Z_r$ be the irreducible components of $Z$.  Then, the restriction of $f$ on $Z_i$ is finite and generically \'etale, and hence also surjective.  Since $Z$ is a disjoint union of $Z_1,\cdots, Z_r$, then we have
$$\sections(Z/U)=\bigsqcup_{i=1}^r \sections(Z_i/U),$$
and
$$\sections(Z_C/C)=\bigsqcup_{i=1}^r \sections((Z_i)_C/C). $$
Therefore, by the case when $Z$ is irreducible, the restriction map of sections is bijective.
\end{proof}

\subsection{Notations and Set up}

In the rest of this paper, we will assume that $S$ is a smooth, quasi-projective $k$-variety of dimension$\ge 2$, where $k$ is an uncountable algebraically closed field of characteristic zero.  And we fix a generically finite dominant morphism $u_0: S\rightarrow \mathbb{P}^n_k$ so that we can talk about lines and line-pairs, or curves and curve-pairs in $S$ (see Definition~\ref{linesdef}).  Note that, without changing any of the results, we will assume further that $u_0$ is a finite, \'etale morphism over a dense, Zariski open subset of $\mathbb{P}^n_k$.  There is a table in Appendix~\ref{tab:TableOfNotationForMyResearch} to help readers keep track of the notations in the rest of this article.

\subsubsection{Bad sets in parameter spaces}

Let $k$ be an algebraically closed field of characteristic zero.  Suppose that $S$ is a smooth variety over $k$ with a finite \'etale morphism onto a dense, Zariski open subset of $\mathbb{P}_k^n$, say, $u_0:S\rightarrow \mathbb{P}_k^n$.  We claim that there exists a smooth projective compactification $W$ of $S$ extending $u_0$.  Since $S$ is quasi-projective, let $W$ be the reduced projective completion of $S$.  Up to replacing $W$ by its normalization, we can assume $W$ is normal.  Then, $W$ is singular only at a codimension two closed subset(\cite{Liuqing}, Prop.4.2.24).  Moreover, since the characteristic of $k$ is zero, by Hironaka's resolution of singularities, there exists smooth $W'$ and $W'\rightarrow W$, which is birational, projective and an isomorphism on the smooth locus of $W$.  So, replacing $W$ by $W'$, we can assume further that $W$ is smooth and projective over $k$, and $S$ is a dense open subset in $W$.  We include the following diagram to clarify the situation.

\[\xymatrix{
S\ar@{^{(}->}[r]^{\begin{subarray}{c} \text{open} \\ \text{dense} \end{subarray}}\ar[d]_{\begin{subarray}{c} \text{finite} \\ \text{\'etale} \end{subarray}} &  W\ar[d]  \\
\text{Image(S)}\ar@{^{(}->}[r]_{\begin{subarray}{c} \text{\,\,\,\,\,\,\,\,\,\,open} \\ \text{\,\,\,\,\,\,\,\,\,\,dense} \end{subarray}} & \mathbb{P}^n_k
}\]

Now, the image of $W$ contains an open dense subset of $\mathbb{P}^n_k$ since $S$ is finite and \'etale over a dense open subset of $\mathbb{P}^n_k$.  Moreover, since $W$ is projective over $\spec k$, $W\rightarrow \mathbb{P}^n_k$ is projective, and hence the image of $W$ is the whole $\mathbb{P}^n_k$.  So the space of conics in $W$ and the space of curve-pairs in $W$ are the same as the space of conics and curve-pairs in $\mathbb{P}^n_k$.

\begin{definitionnotation}\label{projectivecompactification}
The smooth projective variety $W$ constructed as above is called \emph{a projective compactification} of $S$.
\end{definitionnotation}

Recall that an Abelian scheme over $S$ is defined as a proper and smooth $S$-group scheme with connected fibers.  Theorem~\ref{maintheoremjason} gives the result for restriction of sections over line-pairs for a family of Abelian varieties.  We hope to generalize Theorem~\ref{maintheoremjason} to schemes $X$ admitting a finite morphism to some Abelian scheme $A$ over $S$.  Unfortunately, in this situation, the trick of taking boundaries fails to apply on $X$ (cf. \cite{GJ}, Lemma 4.3, Lemma 4.4 and Lemma 4.5).  So the isotrivial factor of $A$ gives moduli of sections (see Remark~\ref{modulibyisotrivialfactor}), and hence we have to consider curves of higher degree instead of line-pairs.  The process of proof will involve the application of pseudo-N\'eron models.

We first fix some notations to clarify the situation.  Let $A$ be an Abelian scheme over $S$, and $f: X\rightarrow A$ be a finite $S$-morphism.  There exists an open dense subset $V\subset S$ and a finite \'etale Galois cover $p: V'\rightarrow V$ such that the pullback of $A$ to $V'$ is isogenous to a product  of a strongly nonisotrivial family of Abelian varieties and a trivial family (see the proof of Theorem 4.7 in \cite{GJ}).  Without changing any results, we can assume that $V=S$.  And denote $V'$ by $S'$.

\begin{definitionnotation}\label{chowtrace}
Let $A_0$ be an Abelian variety over $k$ such that $(A_0,v_0)$ is a Chow $S'/k$-trace of $S'\times_S A$ where $v_0: S'\times_k A_0\to S'\times_S A$ is a morphism of Abelian schemes over $S'$ (\cite{GJ}, Theorem 3.2 (i), p.315).  Then, there exists a strongly nonisotrivial Abelian scheme $Q$ over $S'$ with $v_Q: Q\to S'\times_S A$ a morphism of Abelian schemes, and $v_0\times v_Q: (S'\times_k A_0)\times_{S'}Q \to S'\times_S A$ is an isogeny of Abelian schemes over $S'$ (\cite{GJ}, Cor.3.7, p.317).  Recall that an isogeny of Abelian schemes is a surjective $S$-group morphism with finite fibers, and such an isogeny must be finite.  We denote this isogeny by $\rho_{iso}$.
\end{definitionnotation}

Since $A_0\times_S S'$ is projective over $S'$, the Weil restriction $\mathfrak{R}_{S'/S}(A_0\times_k S')$ exists (\cite{BLR}, Theorem 7.6/4, p.194).  Moreover, since $S$ is a normal scheme, it is geometrically unibranch.  Thus, $A$ is projective over $S$ (\cite{raynaud}, Th\'eor\`eme XI 1.4).  Therefore, the Weil restriction $\mathfrak{R}_{S'/S}(A\times_S S')$ also exists.

\begin{lemma}\label{closedimmersion}
The functorial morphism 
\[\mathfrak{R}_{S'/S}(v_0): \mathfrak{R}_{S'/S}(A_0\times_k S')\to \mathfrak{R}_{S'/S}(A\times_S S')\] 
is a closed immersion.
\end{lemma}

\begin{proof}
Let $\{S_i\}$ be a finite set of \'etale neighborhoods of $S$ such that $\coprod_iS_i\rightarrow S$ is faithfully flat and the base change of $S'\rightarrow S$ by $S_i$ is an open immersion (\cite{BLR}, Prop.2.3/8, p.49).  Denote $\coprod_i S_i$ by $T$.  Since $T\to S$ is faithfully flat and locally of finite presentation, it suffices to prove that $\mathfrak{R}_{S'/S}(v_0)\times\text{Id}_{T}$ is a closed immersion (\cite{fgaexplained}, Prop.1.15, p.9).  However, since Weil restrictions commute with base change, we are reduced to the case where $S'$ is the disjoint union $\coprod_i S_i$.  Therefore, we have the isomorphisms
\[\mathfrak{R}_{S'/S}(A_0\times_k S')=\prod_i\mathfrak{R}_{S_i/S}(A_0\times_kS'\times_{S} S_i)=\prod_i A_0\times_k S'\times_{S}S_i,\]
and
\[\mathfrak{R}_{S'/S}(A\times_S S')=\prod_i\mathfrak{R}_{S_i/S}(A\times_S S'\times_{S} S_i)=\prod_i A\times_S S'\times_{S}S_i\]
(see the proof of Prop.7.6/5 in \cite{BLR}, p.196).  Since $k$ is a field of characteristic zero, the morphism $v_0$ is a closed immersion (\cite{conrad}, p.20 and p.21).  Thus, $\mathfrak{R}_{S'/S}(v_0)$ is a closed immersion.
\end{proof}

\begin{lemmadef}\label{isotrivialfactor}
Let $A\to \mathfrak{R}_{S'/S}(A\times_S S')$ be the functorial morphism of $S$-schemes, which is a closed immersion since $A$ is separated over $S$ (\cite{BLR}, p.197).  The \emph{isotrivial factor} of the Abelian scheme $A$ over $S$ is the fiber product of $A$ and $\mathfrak{R}_{S'/S}(A_0\times_k S')$ over $\mathfrak{R}_{S'/S}(A\times_S S')$.  In other words, the following diagram is Cartesian
\[\xymatrix{
\iso(A) \ar[r]\ar[d]  & \mathfrak{R}_{S'/S}(A_0\times_k S')\ar[d]^{\mathfrak{R}_{S'/S}(v_0)}  \\
A \ar[r] & \mathfrak{R}_{S'/S}(A\times_S S'). 
}\]
Then, $\iso(A)$ is a closed Abelian subscheme of $A$ over $S$.  
\end{lemmadef}

\begin{proof}
For any $S$-scheme $W$, the Weil restriction $\mathfrak{R}_{S'/S}(A\times_S S')$ represents the functor
\[W\mapsto \text{Hom}_{S'}(W\times_S S', A\times_S S').\]
Thus, $\mathfrak{R}_{S'/S}(A\times_S S')$ is a group scheme over $S$ because $A\times_S S'$ is a group scheme.  Moreover, since $A\to \mathfrak{R}_{S'/S}(A\times_S S')$ is induced by the identity on $A\times_S S'$, the functorial morphism $A\to \mathfrak{R}_{S'/S}(A\times_S S')$ is a homomorphism of group schemes.  Similarly, $\mathfrak{R}_{S'/S}(A_0\times_k S')$ is a group scheme over $S$ and $\mathfrak{R}_{S'/S}(A_0\times_k S')\to \mathfrak{R}_{S'/S}(A\times_S S')$ is a homomorphism of group schemes.  Therefore, $\iso(A)$ is a group scheme over $S$.

Let $T$ be the disjoint union of schemes as in Lemma~\ref{closedimmersion}.  Then, $\iso(A)\times_S T$ is $\prod_i A_0\times_k S'\times_{S}S_i$.  Thus, $\iso(A)$ is smooth over $S$ by the standard descent results (\cite{fgaexplained}, Prop.1.15, p.9).  Moreover, since $\iso(A)$ is a closed subscheme of the projective $S$-scheme $A$, $\iso(A)$ is projective over $S$.

Let $b_i$ be a point in $S_i$ and $b$ be the image of $b_i$ in $S$.  Then, $\kappa(b_i)$ is a finite separable extension of $\kappa(b)$.  Since surjectivity is stable under base change and $S'\times_S S_i\to S_i$ is an open immersion, the morphism $S'\times_S S_i\to S_i$ is an isomorphism.  Let $b''$ be a point in $S'\times_S S_i$ whose image in $S_i$ is $b_i$.  Then, $\kappa(b'')$ is the same as $\kappa(b_i)$.  Denote the image of $b''$ in $S'$ by $b'$.  

\[\xymatrix{
S'\ar[d]_{\begin{subarray}{c} \text{finite} \\ \text{\'etale} \end{subarray}} & S'\times_S S_i\ar[l]\ar[d]^{\begin{subarray}{c} \text{open} \\ \text{immersion} \end{subarray}} \\
S & S_i\ar[l]^{\text{\'etale}}
}\]
Then, $\mathfrak{R}_{S'/S}(A\times_S S')\times_S S_i$ equals $\mathfrak{R}_{S'\times_S S_i/S_i}(A\times_S S'\times_S S_i)$ which is $A\times_S S'\times_S S_i$ since $S'\times_S S_i\to S_i$ is an isomorphism.  Therefore, the geometric fiber $\mathfrak{R}_{S'/S}(A\times_S S')_{\overline{b}}$ is equal to 
\[(A\times_S S'\times_S S_i)\times_{S_i}\spec\overline{\kappa(b_i)}\]
which is the same as
\[A\times_S (S'\times_S S')\times_{S'}\spec\overline{\kappa(b'')}.\]
Let $G$ be the Galois group of $S'\to S$ (\cite{BLR}, Example B, p.139).  Then, $S'\times_S S'$ is isomorphic to the disjoint union of $S'$, $G\times S'$.  So the geometric fiber is
\[A\times_S (G\times S')\times_{S'}\spec\overline{\kappa(b'')},\]
i.e., a disjoint union of $|G|$ copies of the geometric fiber $A_{\overline{b}}$.  The same argument gives that the geometric fiber $\mathfrak{R}_{S'/S}(A_0\times_k S')_{\overline{b}}$ is a disjoint union of $|G|$ copies of the Abelian variety $A_0\times_k\spec\overline{\kappa(b)}$.

Because $A\to \mathfrak{R}_{S'/S}(A\times_S S')$ is a closed immersion, this morphism includes the geometric fiber $A_{\overline{b}}$ as one copy of the disjoint union of $|G|$ copies of $A_{\overline{b}}$.  Therefore, the geometric fiber of $\iso(A)$ over $b$ is the Abelian variety $A_0\times_k\spec\overline{\kappa(b)}$, which is irreducible.  As a consequence, $\iso(A)$ is a smooth, projective group scheme over $S$ with connected geometric fibers.  So $\iso(A)$ is an Abelian $S$-scheme.
\end{proof}

\begin{remark}\label{fiberdimension}
Let $b$ be a closed point of $S$.  Then the geometric fiber of $\iso(A)$ at $b$ is just $A_0\times_k\spec \overline{\kappa(b)}$, where $A_0$ is the Chow $S'/k$-trace.  Since $\iso(A)$ is smooth and projective over $S$, all the fibers of $\iso(A)$ over $S$ have the same dimension (\cite{Hart}, Cor.III.9.10, p.263).  Thus, the fiber dimension of $\iso(A)$ is just $\dim A_0$.
\end{remark}

\begin{lemmadef}\label{isotrivialquotient}
There exists a morphism of $S$-schemes $\pi: A\to \iso(A)$ such that the composition
\[\iso(A){\xrightarrow{\hspace*{1cm}}}A\overset{\pi}{\xrightarrow{\hspace*{1cm}}} \iso(A)\]
is an isogeny on the generic fiber of $\iso(A)$.  The morphism $\pi$ is called the \emph{isotrivial quotient} of the Abelian scheme $A$.
\end{lemmadef}

\begin{proof}
Over the function field $K$ of $S$, $\iso(A)_K$ is an Abelian subvariety of $A_K$.  By Poincar\'e's complete reducibility theorem (\cite{isogeny}, Theorem 8.9.3, p.267), there is an Abelian subvariety $B$ of $A_K$ such that the restriction of multiplication gives an isogeny 
\[m: \iso(A)_K\times_K B\to A_K.\]
Since $K$ is a perfect field, there is a dual isogeny
\[\widehat{m}: A_K\to\iso(A)_K\times_K B\]
such that $\widehat{m}\circ m:\iso(A)_K\times_K B\to \iso(A)_K\times_K B$ is the multiplication by $\text{deg}(m)$.  Let $\projection_1$ be the first projection from $\iso(A)_K\times_K B$ to $\iso(A)_K$.  Let $\iota: \iso(A)_K\to A_K$ be the closed immersion.  Then, $(\projection_1\circ\widehat{m})\circ\iota$ is the multiplication by $\text{deg}(m)$ on $\iso(A)_K$, which is an isogeny.

Since the Abelian scheme $\iso(A)$ is a N\'eron model of $\iso(A)_K$ (\cite{BLR}, Prop.1.2/8, p.15), the morphism $\projection_1\circ\widehat{m}: A_K\to \iso(A)_K$ extends to an open dense subset of $S$ containing all codimension one points of $S$ (Corollary~\ref{higherdimensionneron}).  Therefore, we have a rational map $\pi: A\dashrightarrow\iso(A)$.  Since $S$ is regular, the rational map $\pi$ is defined everywhere (\cite{BLR}, Cor.8.4/6, p.234).
\end{proof}

Denote $\rho: \iso(A)\rightarrow S$ as the structure morphism of $\iso(A)$.  If $b$ is a point in $S$, we will denote the fiber of $\iso(A)$ over $b$ by $\iso(A)_b$.  Now we define the bad set for sections and curve-pairs.  Fix a point $b\in S$, and let $p\in \iso(A)$ and $\sigma$ be a section of $A$ over $S$ mapping $b$ to $p$.  Denote $\mathfrak{m}=(m,q,l)$ be a curve-pair in $S$ where $m$ is a smooth curve, $l$ is a conic such that they intersect at $q\in S$.  

\begin{definitionnotation}\label{sectionsdefinition}
Let $\sections_b^p(A/S)$ be the set of sections of $A$ over $S$ such that every section in the set maps $b\in S$ to $p\in\iso(A)$ via $\pi: A\to \iso(A)$. 
\end{definitionnotation}
Consider the following three properties,
\begin{enumerate}[label=(\roman*)]
\item $\sections_b^p(A/S)\rightarrow\sections_b^p((A\times_S\mathfrak{m})/\mathfrak{m})$ is bijective;
\item $\sections_b^p((X\times_{A,\sigma}S)/S)\rightarrow \sections_b^p((X\times_{A, \sigma}S\times_S l)/l)$ is bijective,
\end{enumerate}
where the maps of the sets of sections are restrictions and the fiber product $X\times_{A,\sigma}S$ comes from the section $\sigma$ from $S$ to $A$ mapping $b$ to $p$.

Then, intuitively, the bad set will be 
\[\{(p,\sigma),(m,q,l)|either\,\,(i)\,\,is\,\,false\,\,or\,\,(ii)\,\,is\,\,false\}.\]
We give the rigorous construction of the bad set as following.  From now to the end of this subsection, we fix a point $b\in S$.  

\begin{definitionnotation}\label{grothendieckpifunctor}
Let $\mathfrak{S}\rightarrow \iso(A)$ be an irreducible component of the relative Grothendieck-$\Pi$-scheme parameterizing a point $p$ of $\iso(A)$ and a section $\sigma$ of $A$ over $S$ that maps $b$ to $p$ in the isotrivial quotient $\iso(A)$.   Let $\mathbb{P}^n_{\mathfrak{S}}$, resp. $S_{\mathfrak{S}}$, resp. $A_{\mathfrak{S}}$, resp. $X_{\mathfrak{S}}$, be the base change $\mathfrak{S}\times_k\mathbb{P}^n_k$, resp. $\mathfrak{S}\times_k S$, resp. $\mathfrak{S}\times_k A$, resp. $\mathfrak{S}\times_k X$.  Denote by $\projection_1$ and $\projection_2$ the projections from $S_{\mathfrak{S}}$ to $\mathfrak{S}$ and $S$ respectively.    
\end{definitionnotation}

\begin{definitionnotation}\label{md2}
Let $\mathcal{M}_{d+2}(\mathbb{P}^n_k, \tau)$ be the stack parameterizing pointed curve-pairs of degree $d+2$ in $\mathbb{P}^n_k$.  These are 4-tuples for such a pair $(s, [m], t, [l])$ consisting of a point $s$ of $\mathbb{P}^n_k$, a smooth curve $m$ of degree $d$ that contains $s$, a point $t$ on $m$, and a conic $l$ that contains $t$.  
\end{definitionnotation}

The marked point $s$ defines an evaluation morphism 
\[\rho_{ev}: \mathcal{M}_{d+2}(\mathbb{P}^n_k, \tau) \rightarrow \mathbb{P}^n_k, (s, [m], t, [l])\mapsto s.\]

\begin{definitionnotation}\label{mds}
Denote by $\mathcal{M}_{d+2}({\mathbb{P}^n_{\mathfrak{S}}}{/\mathfrak{S}}, \tau)$ the stack over $\mathfrak{S}$ that sends an $\mathfrak{S}$-scheme $T$ to the set of closed subschemes $\mathcal{C}_T$, flat over $T$, of $\mathbb{P}^n_T\times_T\mathbb{P}^n_T$ such that every geometric fiber of $\gamma_T$ is a pointed curve-pair, $([m\cup_t l], s)$, of a degree $d+2$ curve where $\gamma_T: \mathcal{C}_T\mapsto T$ is the structure morphism of $\mathcal{C}_T$, and $s\in m$.  
\end{definitionnotation}

For every $\mathcal{C}_T$, let $\mathcal{C}_T\rightarrow \mathbb{P}^n_T$ be the composition of the closed immersion $\mathcal{C}_T\rightarrow \mathbb{P}^n_T\times_T\mathbb{P}^n_T$ and the projection of $\mathbb{P}^n_T\times_T\mathbb{P}^n_T$ to its second factor.  Then, this defines the evaluation map 
\[\rho_{ev}^{\mathfrak{S}}: \mathcal{M}_{d+2}({\mathbb{P}^n_{\mathfrak{S}}}{/\mathfrak{S}}, \tau)\rightarrow \mathbb{P}^n_{\mathfrak{S}}.\]
Moreover, take $u_0^{\mathfrak{S}}: S_{\mathfrak{S}}\rightarrow \mathbb{P}^n_{\mathfrak{S}}$ as the morphism $\projection_1\times (u_0\circ\projection_2)$.  The base change of the section mapping $\spec k$ to $b\in S$ gives a section $\lambda$ of $S_{\mathfrak{S}}\rightarrow \mathfrak{S}$, and hence $\mu(b)=u_0^{\mathfrak{S}}\circ \lambda$ is a section of $\mathbb{P}^n_{\mathfrak{S}}$ over $\mathfrak{S}$.  

\[\xymatrix{
\mathfrak{S}\ar[d]_{\text{Id}_{\mathfrak{S}}}\ar@/^1pc/[r]^{\lambda}  &  S_{\mathfrak{S}}\ar[r]_{\projection_2}\ar[l]^{\projection_1}\ar[d]^{u_0^{\mathfrak{S}}}  &  S\ar[d]^{u_0}  \\
\mathfrak{S}  &  \mathbb{P}^n_{\mathfrak{S}}\ar[r]_{\projection_2}\ar[l]^{\projection_1}  &  \mathbb{P}_k^n
}\]

\begin{definitionnotation}\label{mdbb}
Denote by $\mathcal{M}_{d+2}({\mathbb{P}^n_{\mathfrak{S}}}{/\mathfrak{S}}, \tau)_{\mu(b)}$ the fiber product of $\rho_{ev}^{\mathfrak{S}}$ and the section $\mu(b): \mathfrak{S}\rightarrow \mathbb{P}^n_{\mathfrak{S}}$ over $\mathbb{P}^n_{\mathfrak{S}}$.  
\end{definitionnotation}
Then, $\mathcal{M}_{d+2}({\mathbb{P}^n_{\mathfrak{S}}}{/\mathfrak{S}}, \tau)_{\mu(b)}$ parameterizes the set of pairs of sections and curve-pairs $((b, [m], t, [l]), (\sigma, p))$ where $\sigma$ is a section of $A$ over $S$ mapping $b$ to $p$ and the marked point on the curve $[m]$ is the fixed point $b$.

\begin{definitionnotation}\label{md01}
Denote by $\mathcal{M}_{0,d}({\mathbb{P}^n_{\mathfrak{S}}}{/\mathfrak{S}}, 1)$ the stack over $\mathfrak{S}$ sending every $\mathfrak{S}$-scheme $T$ to the set of closed subschemes $\mathcal{B}_T$ of $\mathbb{P}^n_T\times_T\mathbb{P}^n_T$ such that every geometric fiber of $\mathcal{B}_T\rightarrow T$ is a genus-0, degree-$d$ curve with a marked closed point in $\mathbb{P}^n_k$, $([m], s)$.
\end{definitionnotation}

There is a forgetful morphism that sends a pointed curve-pair to the curve $[m]$ and the marked point $s$,
\[\mathcal{M}_{d+2}({\mathbb{P}^n_{\mathfrak{S}}}{/\mathfrak{S}}, \tau)\rightarrow \mathcal{M}_{0,d}({\mathbb{P}^n_{\mathfrak{S}}}{/\mathfrak{S}}, 1), (s, [m], t, [l])\mapsto (s, [m]).\]
Compose this morphism with the projection from $\mathcal{M}_{d+2}({\mathbb{P}^n_{\mathfrak{S}}}{/\mathfrak{S}}, \tau)_{\mu(b)}$ to $\mathcal{M}_{d+2}({\mathbb{P}^n_{\mathfrak{S}}}{/\mathfrak{S}}, \tau)$.  We get an $\mathfrak{S}$-morphism
\[\varphi: \mathcal{M}_{d+2}({\mathbb{P}^n_{\mathfrak{S}}}{/\mathfrak{S}}, \tau)_{\mu(b)}\to\mathcal{M}_{0,d}({\mathbb{P}^n_{\mathfrak{S}}}{/\mathfrak{S}}, 1).\]
We summarize the objects in the following diagram.
\[\xymatrix{
\mathcal{M}_{d+2}({\mathbb{P}^n_{\mathfrak{S}}}{/\mathfrak{S}}, \tau)_{\mu(b)}\ar[r]\ar[d]\ar@/_5pc/[dd]_{\varphi}   &  \mathfrak{S}\ar[d]^{\mu(b)}   \\
 \mathcal{M}_{d+2}({\mathbb{P}^n_{\mathfrak{S}}}{/\mathfrak{S}}, \tau)\ar[r] _{\,\,\,\,\,\,\,\,\,\,\,\,\,\,\,\,\,\,\,\,\rho_{ev}^{\mathfrak{S}}}\ar[d]  &  \mathbb{P}^n_{\mathfrak{S}}  \\
\mathcal{M}_{0,d}({\mathbb{P}^n_{\mathfrak{S}}}{/\mathfrak{S}}, 1)  &  
}\]

\begin{definitionnotation}\label{md0b}
Denote by $\mathcal{M}_{d+2}(\mathbb{P}^n_k, \tau, b)$ the stack parameterizing pointed curve-pairs in $\mathbb{P}^n_k$ such that the marked point on $m$ is $b$. 
\end{definitionnotation}

\begin{definitionnotation}\label{md0bb}
Denote by $\mathcal{M}_{0,d}(\mathbb{P}^n_k, b)$ the stack parameterizing genus-0, degree-$d$ curves with a marked point $b$ in $\mathbb{P}^n_k$.
\end{definitionnotation}

Let $\alpha$ be a $k$-point of $\mathfrak{S}$.  Then, the fiber of the morphism $\varphi$ over $\alpha$ corresponds to a $k$-morphism
\[\varphi_{\alpha}: \mathcal{M}_{d+2}(\mathbb{P}^n_k, \tau, b)\rightarrow \mathcal{M}_{0,d}(\mathbb{P}^n_k, b)\]
which is the forgetful morphism.  So we have the following diagram.
\[\xymatrix{
V_{\beta}\ar@{^(->}[r]^{\begin{subarray}{c} \text{open} \\ \text{dense} \end{subarray}}\ar[rd] & U_{\beta}\ar[r]\ar[d]& \mathcal{M}_{d+2}(\mathbb{P}^n_k, \tau, b)\ar[r] \ar[d]_{\varphi_{\alpha}} & \mathcal{M}_{d+2}({\mathbb{P}^n_{\mathfrak{S}}}{/\mathfrak{S}}, \tau)_{\mu(b)}\ar[d]^{\varphi}  \\
& \spec \beta \ar[r]& \mathcal{M}_{0,d}(\mathbb{P}^n_k, b)\ar[r]\ar[d] & \mathcal{M}_{0,d}({\mathbb{P}^n_{\mathfrak{S}}}{/\mathfrak{S}}, 1)\ar[d]  \\
& & \spec \alpha \ar[r]  &  \mathfrak{S}
}\]

For every $k$-point $\beta$ of $\mathcal{M}_{0,d}(\mathbb{P}^n_k, b)$, the fiber of $\varphi_{\alpha}$ is a Zariski open dense subset $U_{\beta}$ of the variety parameterizing conic curves in $\mathbb{P}^n_k$ that intersect the genus-0, degree-$d$ curve $m$ corresponding to $\beta$.  

\begin{lemma}
For every section $\sigma$ that maps $b$ to $p$ and corresponds to $\alpha$, the restriction of sections
\[\sections((X\times_{A,\sigma}S)/S)\rightarrow \sections((X\times_{A, \sigma}S\times_S l)/l)\]
is bijective for a general conic curve $l$.  
\end{lemma}

\begin{proof}
Take $C$ as a conic curve in Theorem~\ref{bertinisectionsrevised}.
\end{proof}

Therefore, there is a maximal open dense subset $V_{\beta}$ of $U_{\beta}$ such that the restrictions of sections on conic curves are bijective.  Let $\mathcal{V}_{\alpha}$ be the union of $V_{\beta}$ in $\mathcal{M}_{d+2}(\mathbb{P}^n_k, \tau, b)$.  This is an open dense subset of $\mathcal{M}_{d+2}(\mathbb{P}^n_k, \tau, b)$.  Take the union $\mathcal{U}$ of $\mathcal{V}_{\alpha}$ in $\mathcal{M}_{d+2}({\mathbb{P}^n_{\mathfrak{S}}}{/\mathfrak{S}}, \tau)_{\mu(b)}$, which is also open dense in $\mathcal{M}_{d+2}({\mathbb{P}^n_{\mathfrak{S}}}{/\mathfrak{S}}, \tau)_{\mu(b)}$.   We summarize the notations as following.

\[\xymatrix{
V_{\beta}\ar@{^(->}[d]_{\begin{subarray}{c} \text{open} \\ \text{dense} \end{subarray}}\ar[r]  &  \mathcal{V}_{\alpha}\ar@{^(->}[d]_{\begin{subarray}{c} \text{open} \\ \text{dense} \end{subarray}}\ar[r]  &  \mathcal{U}\ar@{^(->}[d]_{\begin{subarray}{c} \text{open} \\ \text{dense} \end{subarray}}  \\
U_{\beta}\ar[r]\ar[d]& \mathcal{M}_{d+2}(\mathbb{P}^n_k, \tau, b)\ar[r] \ar[d]_{\varphi_{\alpha}} & \mathcal{M}_{d+2}({\mathbb{P}^n_{\mathfrak{S}}}{/\mathfrak{S}}, \tau)_{\mu(b)}\ar[d]^{\varphi}  \\
\spec \beta \ar[r]& \mathcal{M}_{0,d}(\mathbb{P}^n_k, b)\ar[r] & \mathcal{M}_{0,d}({\mathbb{P}^n_{\mathfrak{S}}}{/\mathfrak{S}}, 1)
}\]

\begin{definitionnotation}\label{www}
Suppose that $\mathcal{W}$ is the open dense subset of $\mathcal{M}_{d+2}({\mathbb{P}^n_{\mathfrak{S}}}{/\mathfrak{S}}, \tau)$ parameterizing the degree $d+2$ curve-pairs $\mathfrak{m}$ such that the restriction of sections
\[\sections(A/S)\rightarrow\sections((A\times_S\mathfrak{m})/\mathfrak{m})\]
is bijective.  We will prove the existence of such open dense $\mathcal{W}$ in Corollary~\ref{inductionI}.  
\end{definitionnotation}

Let $\mathcal{W}'=\mathcal{U}\cap (\mathcal{W}\times_{\mu(b)}\mathfrak{S})$.  Then, we define the bad set $\mathcal{D}$ as the complement of $\mathcal{W}'$ in $\mathcal{M}_{d+2}({\mathbb{P}^n_{\mathfrak{S}}}{/\mathfrak{S}}, \tau)_{\mu(b)}$.  And, by our construction, $\mathcal{D}$ parameterizes the pointed pairs $\{(\sigma, p), (b, [m], q, [l])\}$ such that either (i) is false or (ii) is false.  We will denote the bad set $\mathcal{D}$ by $\mathcal{D}_b$ since it depends on the choice of $b$ by construction.

\[\xymatrix{
\mathcal{W}'\ar@{^(->}[d]_{\,\,\,\,\,\,\,\,\,\,\,\,\,\,\,\,\begin{subarray}{c} \text{open} \\ \text{dense} \end{subarray}}\ar@{^(->}[r]^{\begin{subarray}{c} \text{open} \\ \text{dense} \end{subarray}\,\,\,\,\,\,\,\,\,\,\,\,\,\,\,\,} & \mathcal{U}\ar@{^(->}[d]^{\begin{subarray}{c} \text{open} \\ \text{dense} \end{subarray}}  &   \\
\mathcal{W}\times_{\mu(b)}\mathfrak{S}\ar@{^(->}[r]_{\begin{subarray}{c} \text{open} \\ \text{dense} \end{subarray}\,\,\,\,\,\,\,\,\,\,\,\,\,\,\,\,}\ar[d] & \mathcal{M}_{d+2}({\mathbb{P}^n_{\mathfrak{S}}}{/\mathfrak{S}}, \tau)_{\mu(b)}\ar[r]\ar[d]   &  \mathfrak{S}\ar[d]^{\mu(b)}   \\
\mathcal{W}\ar@{^(->}[r]_{\begin{subarray}{c} \text{open} \\ \text{dense} \end{subarray}\,\,\,\,\,\,\,\,\,\,\,\,\,\,\,\,\,\,\,\,\,\,} &  \mathcal{M}_{d+2}({\mathbb{P}^n_{\mathfrak{S}}}{/\mathfrak{S}}, \tau)\ar[r] _{\,\,\,\,\,\,\,\,\,\,\,\,\,\,\,\,\,\,\,\,\rho_{ev}^{\mathfrak{S}}}  &  \mathbb{P}^n_{\mathfrak{S}}  
}\]

Note that there are two natural projections from $\mathcal{D}_b$ to $\mathfrak{S}$ and $\mathcal{M}_{d+2}(\mathbb{P}^n_k, \tau, b)$, i.e.
\[\phi_1: \mathcal{D}_b\to \mathfrak{S},\,\,\, by\,\,\{(\sigma, p), (b, [m], q, [l])\}\mapsto (\sigma, p),\]
\[\phi_2: \mathcal{D}_b\to \mathcal{M}_{d+2}(\mathbb{P}^n_k, \tau, b), \,\,\,by\,\,\{(\sigma, p), (b, [m], q, [l])\}\mapsto (b, [m], q, [l]).\]
By projecting once more from $\mathfrak{S}$ to $\iso(A)$, we get a morphism
\[\phi_3: \mathcal{D}_b\to \iso(A),\,\,\, by\,\,\{(\sigma, p), (b, [m], q, [l])\}\mapsto  p.\]
Denote the fiber of $\phi_3$ over $p$ by $\mathcal{D}_b^p$.

\begin{definition}\label{badsets}
The set $\mathcal{D}_b$ constructed above is called \emph{the bad set of sections and curve-pairs marked by $b$}.  For $\{(\sigma, p), (b, [m], q, [l])\}$ in $\mathcal{D}_b^p$, $p$ is called a \emph{bad point} for the curve-pair $(b, [m], q, [l])$, and $\sigma$ is called a \emph{bad section} for $(b, [m], q, [l])$.
\end{definition}

\begin{definition}\label{goodcurves}
Fix $b\in S$ and $p\in\iso(A)_b$ closed points.  A curve-pair $\mathfrak{m}=(b, [m], q, [l])$ with $b$ marked on $[m]$ is called \emph{good for a section $\sigma$ in $\sections_b^p(A/S)$} if the following three properties hold
\begin{enumerate}[label=(\roman*)]
\item $\sections_b^p(A/S)\rightarrow\sections_b^p((A\times_S \mathfrak{m})/\mathfrak{m})$ is bijective,
\item $\sections_b^p((X\times_{A,\sigma}S)/S)\rightarrow \sections_b^p((X\times_{A, \sigma}S\times_S l)/l)$ is bijective,
\end{enumerate}
where the maps of the set of sections are restrictions and the fiber product $X\times_{A,\sigma}S$ comes from the section $\sigma$ from $S$ to $A$.
\end{definition}

\begin{definition}\label{goodirreduciblecurves}
Fix $b\in S$ and $p\in\iso(A)_b$ closed points.  An irreducible smooth curve $C$ with a marked point $b\in S$ is called \emph{good for a section $\sigma$ in $\sections_b^p(A/S)$} if the following two properties hold
\begin{enumerate}[label=(\roman*)]
\item $\sections_b^p(A/S)\rightarrow\sections_b^p((A\times_S C)/C)$ is bijective,
\item $\sections_b^p((X\times_{A,\sigma}S)/S)\rightarrow \sections_b^p((X\times_{A, \sigma}S\times_S m)/m)$ is bijective,
\end{enumerate}
where the maps of the set of sections are restrictions and the fiber product $X\times_{A,\sigma}S$ comes from the section $\sigma$ from $S$ to $A$.
\end{definition}

\begin{remark}\label{markonconic}
$\mathcal{D}_b^p$ could also be defined for curve-pairs with the marked point $b$ on the conic.  Definition~\ref{badsets} and Definition~\ref{goodcurves} are defined in the same way.
\end{remark}

\subsubsection{The space of curve-pairs}

\begin{definitionnotation}\label{mdvb}
Let $\mathcal{H}=\mathcal{M}_{d+2}(\mathbb{P}^n_k, \varepsilon, b)$ be the scheme representing the functor that sends every algebraic $k$-scheme $T$ to the set of closed subschemes $\mathcal{A}_T$ of $\mathbb{P}^n_T$ such that every geometric fiber of $\mathcal{A}_T$ over $T$ is a genus zero, degree $d+2$ curve with a marked point $b$, which can be reducible but at worst a curve-pair.  And, if the curve is a curve-pair, the marked point $b$ is on the irreducible component that is not a conic.  
\end{definitionnotation}

Let $X$ be a scheme over $S$ such that it admits a pseudo-N\'eron model $\widetilde{X}$ over an open dense $\widetilde{S}$ in $S$ of codimension at least two, e.g., $X$ admits a finite morphism to an Abelian scheme $A$ over $S$.

\begin{definitionnotation}\label{xxx}
Delete from $\mathcal{M}_{d+2}(\mathbb{P}_k^n, \tau, b)$ the closed subset representing curve-pairs in which the degree $d$ smooth curve $C_0$ and the conic $C_1$ are tangent to each other or intersect at more than two points.  The let $\mathcal{X}$ be the subspace of $\mathcal{M}_{d+2}(\mathbb{P}_k^n, \tau, b)$ after this deletion.  
\end{definitionnotation}

\begin{definitionnotation}\label{xxx0}
Denote the open locus of genus-$0$, degree-$(d+2)$ curves or curve-pairs contained in $\widetilde{S}$ by $\mathcal{H}_0$ and $\mathcal{X}_0$ respectively.  Let the space of degree-$(d+2)$ curves or curve-pairs in $S$ be $\mathcal{H}_1$ and $\mathcal{X}_1$ respectively.  
\end{definitionnotation}

\begin{definitionnotation}\label{ch01}
Let $\mathcal{C}_{\mathcal{H}}$ be the universal family of degree-$(d+2)$ curve over $\mathcal{H}$, i.e., every geometric fiber of $\mathcal{C}_{\mathcal{H}}$ over $\mathcal{H}$ is a degree-$(d+2)$ curve in $W$.  We denote the open subset of universal family of degree-$(d+2)$ curves in $\widetilde{S}$ (resp. $S$) by $\mathcal{C}_{\mathcal{H}_0}$ (resp. $\mathcal{C}_{\mathcal{H}_1}$).  
\end{definitionnotation}

\begin{definitionnotation}\label{universalsectionlifting}
Let $H$ be the scheme which is universal for the problem of lifting curves in $\mathcal{H}_1$ from $S$ to $X$, \emph{and mapping the point $b$ to $p$}.  
\end{definitionnotation}

Equivalently, we have the following diagram, where $\Phi$ is the composition of the structure morphism of $H$ over $\mathcal{H}_1$ and the open immersion $\mathcal{H}_1\rightarrow \mathcal{H}_0$ and all squares are Cartesian.

\[\xymatrix{
 H\times_{\mathcal{H}_0}\mathcal{C}_{\mathcal{H}_0}\times_{\widetilde{S}}\widetilde{X}\ar[r]\ar[d] &  \widetilde{X}\times_{\widetilde{S}}\mathcal{C}_{\mathcal{H}_0}\ar[r]\ar[d] &  \widetilde{X}\ar[d] \\
H\times_{\mathcal{H}_0}\mathcal{C}_{\mathcal{H}_0}\ar[d]\ar[r] & \mathcal{C}_{\mathcal{H}_0}\ar[r]\ar[d] & \widetilde{S}   \\
H \ar[r]^{\Phi} & \mathcal{H}_0   &
}\]

We note that $H$ is locally of finite type over $\mathcal{H}_0$, and hence over $k$, but may have infinitely many irreducible components.  Every irreducible component is quasi-projective over $\mathcal{H}_0$.  Note that for every closed point $x\in \mathcal{H}_0$ the fiber of $\Phi$, $H_x$, is either empty or a discrete, 0-dimensional variety consisting of at most countably many points (Lemma~\ref{moduliofsections}).  In particular, $H_x$ is a reduced scheme over $k$.

\begin{definitionnotation}\label{universalhhhhhhhh}
There is a universal section of $H\times_{\mathcal{H}_0}\mathcal{C}_{\mathcal{H}_0}\times_{\widetilde{S}}\widetilde{X} \rightarrow H\times_{\mathcal{H}_0}\mathcal{C}_{\mathcal{H}_0}$.  We compose it with the top line of the digram, and denote the morphism by 
\[\varrho: H\times_{\mathcal{H}_0}\mathcal{C}_{\mathcal{H}_0}\rightarrow \widetilde{X},\]
which factors through the inclusion $X\to \widetilde{X}$.  
\end{definitionnotation}

\subsection{Restrictions of Sections for Abelian Schemes}

\subsubsection{Main pseudo-N\'eron model theorem}

\begin{lemma}\label{cartiersmooth}
Let $X\rightarrow S$ be a morphism locally of finite type of regular Noetherian schemes.  Let $Z$ be a codimension one regular closed subscheme of $X$, and suppose that $Z\rightarrow S$ is smooth.  Then, there exists an open subset $U$ of $X$ that contains $Z$ such that $U\rightarrow S$ is smooth.
\end{lemma}

\begin{proof}
See Appendix~\ref{Proofoftwolemmas}.
\end{proof}

The following theorem is the key application of pseudo-N\'eron models in the problem of restriction of sections.  The proof is exactly the same as Lemma 4.13 of \cite{GJ}.  We prove the relative version of this theorem, i.e., for fixed $b\in S$ and $p\in\iso(A)_b$.

\begin{theorem}\label{keylemma}
Suppose that:
\begin{itemize}
\item{$X$ is smooth projective over $S$,}
\item{$X$ has a pseudo-N\'eron model $\widetilde{X}$ over $W$, and}
\item{every geometric fiber $X_{\overline{s}}$ does not contain any rational curve.}
\end{itemize}
Then, any irreducible component $H_0$ of $H$ which dominates $\mathcal{H}$ also dominates $\mathcal{X}$.  That is, the intersection of the image $\Phi(H_0)$ with $\mathcal{X}$ contains a dense open subset of $\mathcal{X}$.
\end{theorem}

\begin{proof}
As in \cite{GJ}, we consider the diagram:

\[\xymatrix{
\mathcal{C}_{\mathcal{X}}\ar[r]\ar[d] & \mathcal{C}_{\mathcal{H}}\ar[r]\ar[d] &  W\\
\mathcal{X}\ar[r] & \mathcal{H}  &
}\]
where $\mathcal{C}_{\mathcal{X}}$ is the universal family of curve-pairs with nodes of curve-pairs deleted so that $\mathcal{C}_{\mathcal{X}}\rightarrow \mathcal{X}$ is smooth.    Also, the morphisms $\mathcal{C}_{\mathcal{H}}\rightarrow W$ and the composition $\mathcal{C}_{\mathcal{X}}\rightarrow W$ are smooth.

Since $H_0$ is quasi-projective over $\mathcal{H}_0$, we can choose a compactification $\overline{H}$ of $H_0$ such that $\overline{H}$ is normal and $\Phi$ extends to $\overline{\Phi}: \overline{H}\rightarrow \mathcal{H}$.  Let us first take an irreducible projective completion of $H_0$, say $\widetilde{H}_0$, give it a reduced closed subscheme structure and then take its normalization.  As a consequence, $\widetilde{H}_0$ is an integral scheme and the normalization morphism $\nu: \overline{H}\rightarrow \widetilde{H}_0$ is finite.  As topological spaces $H_0$ is a Zariski open dense in $\widetilde{H}_0$.  And $\overline{H}$ is birational and surjective onto $\widetilde{H}_0$.  Denote the inverse image of $H_0$ in $\overline{H}$ by $\overline{H}_0$, which is a normal, integral, open dense subscheme of $\overline{H}$.  We include the following diagram to clarify the situation, where $H_0$ is a locally closed subscheme of $\mathbb{P}^N_{\mathcal{H}}$ for some integer $N$.  

\[\xymatrix{
\overline{H}\ar[r]^{\nu}  &  \widetilde{H}_0\ar@{^{(}->}[r]^{closed}  &  H_0'\ar@{^{(}->}[rr]^{closed} &  &  \mathbb{P}^N_{\mathcal{H}}\ar[d]  \\
\overline{H}_0\ar[rr]\ar@{^{(}->}[u]|-{open}  &  &  H_0\ar@{^{(}->}[u]|-{open}\ar[r]^{\Phi} &  \mathcal{H}_0\ar@{^{(}->}[r]^{open}  &  \mathcal{H}
}\]

Suppose that $H_0$ is an open subscheme of $H_0'$, a closed subscheme of $\mathbb{P}^N_{\mathcal{H}}$, with reduce structure $\widetilde{H}_0$.  Then the composition of the top line and the rightmost will be the morphism $\overline{\Phi}$.  We note that the induced morphism $\overline{H}_0\rightarrow H_0$ can be obtained by first giving a reduced structure to $H_0$ and then taking its normalization.

By construction, $\overline{H}$ is proper over $\mathcal{H}$.  Now, the image of $\overline{H}$ in $\mathcal{H}$ is closed.  So the image is the whole $\mathcal{H}$ since $H_0$ dominates $\mathcal{H}$.  Thus, $\overline{\Phi}^{-1}(\mathcal{X})$ is nonempty.  Moreover, since $\overline{\Phi}^{-1}(\mathcal{X})$ is pure of codimension one in $\overline{H}$ and the non-regular locus of $\overline{H}$ is of codimension two (\cite{Liuqing}, Prop.4.2.24), $\overline{H}$ is regular at a general point of $\overline{\Phi}^{-1}(\mathcal{X})$.  We note that the regular locus and the smooth locus of $\overline{H}$ coincide because $k$ is algebraically closed (\cite{Liuqing}, Cor.4.3.33).  

We have that $\overline{\Phi}: \overline{H}\rightarrow \mathcal{H}$ is a surjective morphism of finite type $k$-schemes, so by the generic smoothness theorem ($\chara k=0$) there exists an open dense subset of the regular locus of $\overline{H}$ on which the morphism $\overline{\Phi}$ is smooth.  Hence, there exists an open dense subset of $\overline{\Phi}^{-1}(\mathcal{X})=\mathcal{X}\times_{\mathcal{H}}\overline{H}$ on which the morphism $\overline{\Phi}_{\mathcal{X}}$, the base change of $\overline{\Phi}$ to $\mathcal{X}$, is smooth.  By checking the inverse image of the smooth locus of $\mathcal{X}$, the reduced (actually smooth) locus of $\overline{\Phi}^{-1}(\mathcal{X})$ is nonempty.  Note that the smooth locus of $\mathcal{X}$ is nonempty by the Jacobian criterion.  Set $\overline{H}_{\mathcal{X}}=\overline{\Phi}^{-1}(\mathcal{X})_{red}$ as the reduced locus of $\overline{\Phi}^{-1}(\mathcal{X})$.  By the generic smoothness theorem again, there exists a dense open $V\subset \overline{H}_{\mathcal{X}}\cap \overline{H}^{reg}$ such that $\overline{\Phi}_{\mathcal{X}}: \overline{H}_{\mathcal{X}}\rightarrow \mathcal{X}$ is smooth on $V$.  

\[\xymatrix{
V\ar@{^{(}->}[r]  &  \overline{\Phi}^{-1}(\mathcal{X})\ar[r]^{\overline{\Phi}_{\mathcal{X}}}\ar@{^{(}->}[d]|-{immersion} & \mathcal{X}\ar@{^{(}->}[d]|-{immersion} \\
  & \overline{H}\ar[r]^{\overline{\Phi}} & \mathcal{H}  
}\]

We summarize the objects in our argument in the following diagram (cf. \cite{GJ}, p.324).  

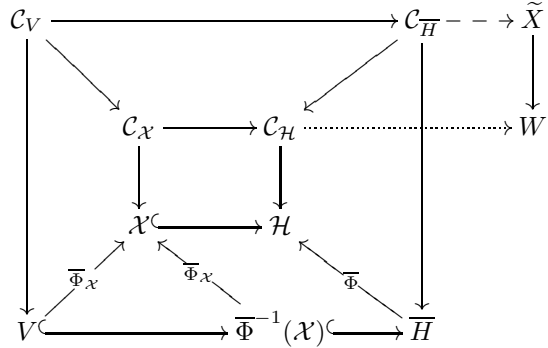
\[\xymatrix{
\mathcal{C}_V\ar[ddd]\ar[rrr]\ar[dr]  &    &    & \mathcal{C}_{\overline{H}}\ar@{-->}[r]\ar[ddd]\ar[dl]  &   \widetilde{X}\ar[d]  \\
  &  \mathcal{C}_{\mathcal{X}}\ar[r]\ar[d]  &  \mathcal{C}_{\mathcal{H}}\ar@{.>}[rr]\ar[d]  &  &  W  \\
  &  \mathcal{X}\ar@{^{(}->}[r]  &  \mathcal{H} & &   \\
V\ar@{^{(}->}[rr]\ar[ur]|-{\overline{\Phi}_{\mathcal{X}}}  &  &  \overline{\Phi}^{-1}(\mathcal{X})\ar@{^{(}->}[r]\ar[ul]|-{\overline{\Phi}_{\mathcal{X}}}   & \overline{H}\ar[lu]|-{\overline{\Phi}} &  
}\]
Denote the base change $\mathcal{C}_{\mathcal{H}}\times_{\mathcal{H}}\overline{H}$ by $\mathcal{C}_{\overline{H}}$.  Then, $\mathcal{C}_V\rightarrow W$ is smooth, $\mathcal{C}_V\rightarrow \mathcal{C}_{\overline{H}}$ is an immersion and $V$ is contained in the regular locus of $\overline{H}$.

We claim that there exists an open subset $U$ of $\mathcal{C}_{\overline{H}}$ containing $\mathcal{C}_V$ such that $U\rightarrow W$ is smooth.  Let $\overline{\mathcal{C}}$ be the locus with double lines and nodes deleted.  Then $\overline{\mathcal{C}}\rightarrow \mathcal{H}$ is smooth.  Note that $\mathcal{C}_{\mathcal{X}}$ is a hypersurface in $\overline{\mathcal{C}}$.  Consider the open subset $\Omega=\overline{\mathcal{C}}\times_{\mathcal{H}}\overline{H}$ of $\overline{H}$.  Thus, $\Omega\rightarrow \overline{H}$ is smooth.  As a consequence, $\Omega$ is regular over the regular locus of $\overline{H}$ (\cite{Liuqing}, Theorem 4.3.36, p.142).  From the construction, $\mathcal{C}_V\subset \Omega^{reg}$ is a locally closed subset.  By Lemma~\ref{cartiersmooth}, there exists an open subset $U$ of $\Omega^{reg}$ containing $\mathcal{C}_V$ such that $U\rightarrow W$ is smooth.  

We also denote the restriction of $\varrho$ on $H_0\times_{\mathcal{H}_0} \mathcal{C}_{\mathcal{H}_0}$ and its base change to $\overline{H}_0\times_{\mathcal{H}_0} \mathcal{C}_{\mathcal{H}_0}$ by $\varrho$.    Then, $\varrho$ is a rational map from $U$ to $X$ over $W$, which is well-defined on $(\overline{H}_0\times_{\mathcal{H}_0} \mathcal{C}_{\mathcal{H}_0})\cap U$.  This rational map is marked as the dashed arrow from $\mathcal{C}_{\overline{H}}$ to $\widetilde{X}$ in the diagram above.
 
Now, let $\widetilde{W}\subset W$ be the image of $U\rightarrow W$ which is an open subset of $W$.  Then, $\widetilde{W}$ contains an open dense subset of the image of $\mathcal{C}_{\mathcal{X}}$ in $W$ since $U$ contains $\mathcal{C}_V$. Since $\widetilde{X}$ is a pseudo-N\'eron model of $X$ over $S$, $U\dashrightarrow\widetilde{X}$ is well-defined outside a codimension two subset of $\widetilde{W}$ by the weak extension property.  Every such codimension two subset in $\widetilde{W}$ can be avoided by a general smooth curve or curve-pair.  This gives an open dense subset $U'$ of $U$ such that $U'$ contains an open dense subset of $\mathcal{C}_{\mathcal{X}}$ and $\varrho$ is well-defined on $U'$.  As a consequence, every curve-pair in an open of $\mathcal{C}_{\mathcal{X}}$ with nodes deleted can be lifted.  Moreover, since there is no rational curve on $X$, any morphism from a punctured curve-pair to $X$ can be extended to the node.  Since $H$ parameterizes the space of degree-$(d+2)$ curves that can be lifted, an open subset of $V$ is contained in the image of $H_0$, i.e. $H_0$ also dominates $\mathcal{X}$.
\end{proof}

\subsubsection{Inductive pseudo-N\'eron deforming step}

Now we prove the existence of the open dense subset $\mathcal{W}$ of $\mathcal{M}_{d+2}({\mathbb{P}^n_{\mathfrak{S}}}/\mathfrak{S}, \tau)$.  That is, the open subet parameterizing the degree-$(d+2)$ curve-pairs $\mathfrak{m}$ such that the restriction of sections
\[\sections(A/S)\rightarrow\sections((A\times_S\mathfrak{m})/\mathfrak{m})\]
is bijective.  In \cite{GJ} (Theorem~\ref{maintheoremjason}), it was proved that this open dense subset exists if $\mathfrak{m}$ is a line-pair.  And since we work over a field of characteristic zero, the same is true for a very general conic, i.e., genus-$0$ and degree-$2$ curves by using the Pseudo-N\'eron model (Theorem~\ref{maintheoremjason}).  The following inductive step gives the relative version of Theorem~\ref{maintheoremjason} for very general genus-0, degree-$(d+2)$ curves, and curve-pairs, and fixed $b\in S$, $p\in\iso(A)_b$.

\begin{corollary}\label{inductionI}
Fix $b\in S$ and $p\in\iso(A)_b$ closed points.  Suppose that for a very general genus-0, degree-$(d+2)$ curve-pair $C\cup m$ every section in $\sections_b^p(X_{C\cup m}/C\cup m)$ is the restriction of a unique section in $\sections_b^p(X/S)$.  Then,  for a very general genus-0, degree-$(d+2)$ irreducible smooth curve contianing $b$, every section  over this curve mapping $b$ to $p$ is the restriction of a unique section of $X$ over $S$.
\end{corollary}

\begin{proof}
The proof comes from \cite{GJ}, Theorem 4.15.  Let $\mathcal{Y}$ be the scheme parameterizing rational sections of $X$ over $S$ mapping $b$ to $p$.  Let $\mathcal{H}_1$ and $H$ be as Theorem~\ref{keylemma}.  Base change $X$, $S$ and the universal family of genus-0, degree-$(d+2)$ curves in $S$ by $\mathcal{Y}$, we have the following diagram,
\[\xymatrix{
  &  A_{\mathcal{Y}}\ar[d]  \\
\mathcal{C}_{\mathcal{H}_1}\times_k \mathcal{Y} \ar[r]\ar[d]\ar@{.>}[ru]^{\varphi}  &  S_{\mathcal{Y}} \\
\mathcal{H}_1\times_k\mathcal{Y} &  
}\]
where the rational map $\varphi$ is the composition of the morphism $\mathcal{C}_{\mathcal{H}_1}\times_k \mathcal{Y}\rightarrow S_{\mathcal{Y}}$ and the universal rational section of $A_{\mathcal{Y}}$ over $S_{\mathcal{Y}}$.  Let $(\mathcal{C}_{\mathcal{H}_1}\times_k\mathcal{Y})^{\circ}$ be the maximal domain of definition of $\varphi$.  The image of $(\mathcal{C}_{\mathcal{H}_1}\times_k\mathcal{Y})^{\circ}$ is also an open subset which will be denoted by $(\mathcal{H}_1\times_k\mathcal{Y})^{\circ}$.  And, similarly, there is an open dense $S_{\mathcal{Y}}^{\circ}$ in $S_{\mathcal{Y}}$.  Then, $\varphi$ on $(\mathcal{C}_{\mathcal{H}_1}\times_k\mathcal{Y})^{\circ}$ gives the universal section of lifting genus-0, degree-$(d+2)$ curves in $S_{\mathcal{Y}}^{\circ}$.  Since $H\times_k \mathcal{Y}$ is universal for the problem lifting genus-0, degree-$(d+2)$ curves in $S_{\mathcal{Y}}$, there is a map from $(\mathcal{H}_1\times_k\mathcal{Y})^{\circ}$ to $H\times_k\mathcal{Y}$.  Thus, we get a rational map $\Phi: \mathcal{H}_1\times_k\mathcal{Y}\dashrightarrow H\times_k\mathcal{Y}$.  Let $\mathcal{X}^{\circ}$ be the very general subset of $\mathcal{X}$ such that the statement about restriction of sections holds on $\mathcal{X}^{\circ}$.  Then,  an open dense of $\mathcal{X}^{\circ}\times_k\mathcal{Y}$ is in the maximal domain of definition of $\Phi$, and hence the projection of the image of $\Phi$ to $H$ is a union of some irreducible components of $H$.

Suppose that the theorem is false for very general genus-0, degree-$(d+2)$ curves.  Then there is an irreducible component, $H_2$, of $H$ dominating $\mathcal{H}_1$ such that the universal section over this irrreducible component is not the restriction of a section over $S$.  By Theorem~\ref{keylemma}, $H_2$ also dominates $\mathcal{X}$.  Thus, there exists an irreducible component $H_3$ of $H$ which is in the projection of the image of $\Phi$ to $H$ such that it intersects with $H_2$.  Then, over a very general point of $\mathcal{X}$, the fiber of $H$ has intersecting irreducible components.  However, this contradicts Cartier's theorem (\cite{mfd}, Theorem 1, Lecture 25) for locally finite type group schemes over a field of characteristic zero.  Therefore, the theorem also holds for a very general genus-0, degree-$(d+2)$ curve in $S$.
\end{proof}

\begin{lemma}\label{verygeneralmatchingabeliancase}
Suppose that every section of $A$ over a very general genus-0, degree-$d$, irreducible, smooth curve is contained in a unique section of $A$ over $S$.  Then, for very general points $b'$ and $b''$ in $S$ and for a very general conic $C$ containing $b'$ and a very general genus-0, degree-$d$ curve $m$ containing $b''$ with $d\ge 2$ such that $C$ and $m$ intersect at a very general point $c$, every section in $\sections(A_{C\cup m}/C\cup m)$ is the restriction of a unique section in $\sections(A/S)$. 
\end{lemma}

\begin{proof}
Denote by $M'$ the space of conics in $S$ and $M''$ the space of genus-0, degree-$d$ curves.  Denote by $C'$, resp. $C''$, the universal family of curves over $M'$, resp. $M''$ (Definition~\ref{defn-witness} (i)).  Let $\gamma$ be a section in $\sections(A_{C\cup m}/C\cup m)$.  Then, by Theorem~\ref{maintheoremjason}, $\gamma|_{C}$ is contained in a unique global section $\sigma\in \sections(A/S)$, and the same is true for $\gamma|_{m}$ by hypothesis.  Let $\tau\in\sections(A/S)$ be the unique global section extending $\gamma|_{m}$.  Denote by $r$ the integer $2\dim(S)-2$.  The transversal Grassmannian $G'$ is the Grassmannian parameterizing $r$-dimensional linear subvarieties $N'$ of $M'$, and similarly define $G''$ for $M''$ (Definition~\ref{defn-witness} (vii)).

For very general $C$ and $m$, they are contained in some very general linear subvariety $N'$ of $M'$ and linear subvariety $N''$ of $M''$ respectively.  Denote by $u_S'$ and $u_S''$ the composed morphisms
\[C'\to S\times_{\spec k}M'\to S,\]
and
\[C''\to S\times_{\spec k}M''\to S\]
(Definition~\ref{defn-witness} (iii)).  The base change of $\sigma$, resp. $\tau$, to $C'$, resp. $C''$, gives a $u_S'$-multisection $\sigma_1$, resp. $u_S''$-multisection $\tau_1$ of $A\to S$ (Definition~\ref{defn-omega}).  Consider the 2-pointed bi-gon $(C'_{t'}\cup_c C''_{t''}, b', b'')$ parameterized by a point $t'\in N'$, resp. $t''\in N''$, whose curve $C'_{t'}$ contains $c$ and $b'$, resp. whose curve $C''_{t''}$ contains $c$ and $b''$ (see the statement of Lemma~\ref{lem-transversal3}).  Denote the image of $\sigma_1$, resp. $\tau_1$ in $C'\times_S Y$, resp. $C''\times_S Y$ by $\Omega'$, resp. $\Omega''$ (Definition~\ref{defn-omega}).

The following composition 
\[\Omega'\times_{N'}C'_{N'}\to A\times_S(C'_{N'}\times_{N'} C'_{N'})\overset{Id_A\times u_{N',S}^{[2]}}{\xrightarrow{\hspace*{2cm}}} A\times_S(S\times_{\spec k}S)\]
defines a $(S\times_{\spec k}S, \projection_1)$-multisection of $A\to S$ (Lemma~\ref{lem-transversal4}).  Let the image of the restriction of this multisection on the fiber $\projection_1^{-1}(b')$ in $A$ be $\Omega'_{M', N', b'}$, and similarly define $\Omega''_{M'', N'', b''}$ (see the notations in Lemma~\ref{lem-transversal3}).  The base change of $\sigma$ via the composition
\[C'_{N'}\times_{N'}C'_{N'}\overset{u_{N',S}^{[2]}}{\xrightarrow{\hspace*{1.5cm}}} S\times_{\spec k}S \overset{\projection_1}{\xrightarrow{\hspace*{1cm}}} S\]
gives a section of $A\times_S(C'_{N'}\times_{N'} C'_{N'})$ over $C'_{N'}\times_{N'}C'_{N'}$, which is the same as the base change of $\sigma$ via 
\[C'_{N'}\times_{N'}C'_{N'}\overset{\projection_1}{\xrightarrow{\hspace*{1cm}}} C'_{N'} \overset{u'_S}{\xrightarrow{\hspace*{1cm}}} S.\]
Thus, the image of $\sigma$ in $A$ equals the image of $\Omega'\times_{N'}C'_{N'}$ in $A$.  The restriction of $\Omega'\times_{N'}C'_{N'}$ on the $\projection_1$-fiber $\projection_1^{-1}(b')$ is the restriction of $\sigma$ on the curves $C'_{t'}$ containing $b'$ and parameterized by $t'\in N'$.  Therefore, the image of $\gamma|_{C}$ is contained in $\Omega'_{M', N', b'}$, and the same argument shows that the image of $\gamma|_{m}$ is contained in $\Omega''_{M'', N'', b''}$.  By the Bi-gon Lemma (Lemma~\ref{lem-transversal3}), for a very general pair $(N', N'', b', c, b'')$ in $G'\times_{\spec k} G''\times_{\spec k}S\times_{\spec k}S\times_{\spec k}S$, $\gamma$ comes from a unique section of $A$ over $S$.
\end{proof}

\begin{corollary}\label{inductiveII}
Let $S$ be a smooth, quasi-projective $k$-scheme of dimension $b\ge 2$.  Let $A$ be an Abelian scheme over $S$.  For a very general curve-pair $C\cup m$ in $S$ such that the degree of $m$ is even, the restriction map of sections
\[\sections(A/S)\to \sections(A_{C\cup m}/C\cup m)\]
is a bijection.  This also holds for $C$ a very general genus-0, irreducible, smooth curve of even degree in $S$.
\end{corollary}

\begin{proof}
Take $X=A$ in Corollary~\ref{inductionI}.  In Lemma~\ref{verygeneralmatchingabeliancase}, take $m$ as a conic curve.  Then by Theorem~\ref{maintheoremjason} and Lemma~\ref{verygeneralmatchingabeliancase} the result holds for very general $C\cup m$.  Next, use Corollary~\ref{inductionI} to deform $C\cup m$ to a very general genus-0, irreducible smooth curve of degree 4.  Attach a very general conic to this curve at a very general point and apply Lemma~\ref{verygeneralmatchingabeliancase} and Corollary~\ref{inductionI} again.  Then, the corollary follows by induction.
\end{proof}

\subsection{Moduli of bad points caused by $\iso(A)$}

\begin{lemma}\label{abelianvarieties}
Let $A$ and $B$ be two Abelian varieties over a field $k$.  Then, there are at most countably many homomorphism of Abelian varieties from $A$ to $B$.
\end{lemma}

\begin{proof}
See Appendix~\ref{Proofoftwolemmas}.
\end{proof}

\begin{lemma}\label{moduliofsections}
Let $C$ be a smooth curve in $S$.  Then, for fixed $b\in C$ and $p\in A_b$ closed points, there are at most countably many sections of $X$ (resp. $A$, resp. $\iso(A)$) over $C$ that map $b$ to $p$.  And there are at most countably many sections of $X$ (resp. $A$, resp. $\iso(A)$) over $S$ that map $b$ to $p$.  In particular, for every $p'\in X_{b}$, there are at most countably many sections of $X$ over $C$, resp. over $S$, mapping $b$ to $p'$.
\end{lemma}

\begin{proof}
It suffices to prove the statement for the Abelian scheme $A$ since $X\to A$ is finite and $\iso(A)$ is a closed subscheme of $A$.

First suppose that $A=A_0\times_k S$ for some abelian variety $A_0$ over $k$.  Since the inclusion $C\rightarrow S$ is fixed, giving a morphism from $C$ to $A_0\times_k S$ is the same as giving a morphism from $C$ to $A_0$.  Up to a translation we can assume that the image of $p$ is the identity in $A_0$.  Then, this is equivalent to specifying a homomorphism $(\jac(C),0)$ to $(A_0,0)$, where $\jac(C)$ is the Jacobian of $C$.  By Lemma~\ref{abelianvarieties}, there are at most countably many such homomorphisms, and hence at most countably many sections from $C$ to $A$ that map $b$ to $p$.

Now, suppose that $A$ is not a trivial family of Abelian varieties.  However, by a finite, \'etale and Galois base change $S'\rightarrow S$, we have an isogeny of Abelian $S'$-schemes,
\[\rho_{iso}: (A_0\times_k S')\times_{S'} Q\to A\times_S S'\]
where $A_0\times_k S'$ is a trivial family of Abelian varieties over $S'$ and $Q$ is a strongly nonisotrivial Abelian scheme over $S'$.

Consider the product $(A_0\times_k S')\times_{S'} Q$.  Let $b'\in S'$ and $p'\in A_0\times_k S'$.  Every section of $(A_0\times_k S')\times_{S'} Q$ over $C'=C\times_S S'$ that maps $b'$ to $p'$ comes from a section of $A_0\times_k S'$ over $C'$ mapping $b'$ to $p'$ and a section of $Q$ over $C'$.  There are at most countably many sections of $A_0\times_k S'$ over $C'$ mapping $b'$ to $p'$.  And since $Q$ is strongly nonisotrivial, there are at most countably many section of $Q$ over $S'$ (\cite{GJ}, Lemma 3.6, p.316).  Thus, there are at most countably many sections of $(A_0\times_k S')\times_{S'} Q$ over $C'$ mapping $b'$ to $p'$.  

By the standard descent result, $\homo_S(C, A)\to \homo_{S'}(C', A\times_S S')$ is injective (\cite{BLR}, Theorem 6.1/6 (a), p.135).  And, if a morphism from $C$ to $A$ is an immersion after the base change by $S'$, so is the original morphism (\cite{EGA}, $\operatorname{IV_2}$, Prop.2.7.1).  So the problem reduces to counting the sections of $A\times_S S'$ over $C'$ mapping a fixed point $b'\in S'$ to a fixed point $p'\in A\times_S S'$.  Let 
\[\tau_{iso}: A\times_S S'\to (A_0\times_k S')\times_{S'} Q\]
be the dual isogeny of $\rho_{iso}$.  Let $p''$ be the image of $p'$ under $\tau_{iso}$.  Then, since $\tau_{iso}$ is finite, for every section $\sigma''$ in $\sections_{b'}^{p''}((A_0\times_k S')\times_{S'} Q/C')$, there are at most finitely many sections of $A\times_S S'$ over $C'$ lifting $\sigma''$ and mapping $b'$ to $p'$.  However, since $\sections_{b'}^{p''}((A_0\times_k S')\times_{S'} Q/C')$ is at most countable, $\sections_{b'}^{p'}(A\times_S S'/C')$ is at most countable.  Putting all these together, $\sections_b^p(A/C)$ is at most countable.  

Replacing the Jacobian of a smooth curve by the Albanese variety of $S$ (\cite{rationalpoints}, Theorem 5.7.13, p.141), the result for $\sections_b^p(X/S)$, resp. $\sections_b^p(A/S)$, resp. $\sections_b^p(\iso(A)/S)$ follows immediately.
\end{proof}

\begin{lemma}\label{matching}
(1). Fix a point $b\in S$, a point $p$ in $\iso(A)_b$, and a point $p'\in (\pi\circ f)^{-1}(p)$.  Let $\sigma$ be a section in $\sections_b^p(A/S)$.  Then, for a conic $C$ and a genus-0, degree-$d$ curve $m$ containing $b$ with $d\ge 2$ such that $C$ and $m$ intersect at a very general point, every section in $\sections_b^{p'}(X_{C\cup m}/C\cup m)$ that maps to $\sigma|_{C\cup m}$ is the restriction of a unique section in $\sections_b^{p'}(X/S)$ that maps to $\sigma$ if $C\cup m$ is good for $\sigma$.

(2). Conversely, if $p$ is a bad point for $C\cup m$, then for $C$ and $m$ intersecting at a very general point, there exists a section in $\sections_b^{p'}(X_{C\cup m}/C\cup m)$ that cannot be extended uniquely.

(3). Let $C$ be a genus-0, degree-$d$, irreducible smooth curve marked by $b$ with $d\ge 2$.  Suppose that $C$ contains another very general point $c$.  Then, every section in $\sections_b^{p'}(X_{C}/C)$ that maps to $\sigma|_{C}$ is the restriction of a unique section in $\sections_b^{p'}(X/S)$ that maps to $\sigma$ if $C\cup m$ is good for $\sigma$.
\end{lemma}

\begin{proof}
First, suppose that $C\cup m$ is good for the section $\sigma$.  Let $f_{\sigma}: X\times_{A,\sigma} S$ be the finite morphism arising from base change of $f$ by $\sigma$.  Denote $g_{\sigma}: X_S\to X$ to be the base change of $\sigma$ by $f$.  For $\sigma\in \sections_b^p(A/S)$, there are at most finitely many sections of $f_{\sigma}$ that map $b$ to $p$.  For any two different such sections, the intersection in $X_S$ maps in $S$, via $f_{\sigma}$, to a proper closed subset of $S$.  As we vary these sections of $f_{\sigma}$, this gives finitely many proper closed subsets of $S$.  Moreover, there are at most countably many sections in $\sections_b^p(X_{m}/m)$.  For any two distinct such sections, the intersection in $X$ maps in $S$, via the structure morphism of $X$, to a proper closed subset of $S$.   The complement of these closed subsets is a very general open subset of $S$.  Denote this very general subset by $S_0(\sigma)$.

Take $c\in S_0$ as the intersection point of $C$ and $m$.  Let $\gamma\in \sections_b^{p'}(X_{C\cup m}/C\cup m)$ such that $(f\circ\gamma)|_m$ is contained in the section $\sigma$ in $\sections_b^p(A/S)$.  Then form the following diagram.

\[\xymatrix{
C\cup m\ar@/_/[ddr]_{\gamma} \ar@/^/[drr]\ar[dr]|-{\gamma_0}&  &  &  \\
& X_S\ar[r]_{f_{\sigma}}\ar[d]^{g_{\sigma}}  &  S\ar[d]^{\sigma} &  \\
& X\ar[r]_f  &  A\ar[r]^{\pi}\ar[d]  & \iso(A)\ar[ld]^{\rho}  \\
&  &   S  &  
}\]

Since $X_S$ is the fiber product of $X$ and $S$ via $f$ and $\sigma$, every section of $f_{\sigma}$ over $S$ (resp. over $C\cup m$) arises from a unique section of $X$ over $S$ (resp. over $C\cup m$).  Thus, $\gamma$ gives a section, $\gamma_0$, of $f_{\sigma}$ over $C\cup m$ such that $g_{\sigma}\circ\gamma_0=\gamma$.  Since $C\cup m$ is good for $\sigma$, $\gamma_0|_C$ is contained in a unique section of $f_{\sigma}$ over $S$, say, $\tau$.  Then, $g_{\sigma}\circ \tau$ is a section of $X$ over $S$ such that $f\circ g_{\sigma}\circ\tau=\sigma$.  By construction, $(g_{\sigma}\circ \tau)|_C$ equals $g_{\sigma}\circ (\gamma_0|_C)$, which is $\gamma|_C$.  Let $\gamma_0$ be the restriction of $g_{\sigma}\circ\tau$ on $m$.  Suppose that $\gamma_0\neq \gamma|_m$, then $\gamma_0(c)=g_{\sigma}\circ\tau(c)=\gamma(c)$ is in the intersection of the images of $\gamma_0$ and $\gamma|_m$.  Since $\gamma_0$ is a section of $X$ over $m$ mapping $b$ to $p$, also $c$ is in the complement of $S_0$, contradicting the choice of $c$.  Therefore, $\gamma_0$ equals $\gamma|_m$, and $\gamma$ extends to a unique section of $X$ over $S$.

Next, suppose that $p$ is a bad point for $C\cup m$.  For every two distinct sections in $\sections_b^p(A/S)$, the intersection of their images in $A$ maps to a proper closed subset of $S$.  Remove these countably many closed subsets from $\cap_{\sigma}S_0(\sigma)$, $\sigma\in\sections_b^p(A/S)$.  Denote this very general subset by $S^{\circ}$.  Take $c\in S^{\circ}$ and a section $\tau\in\sections_b^p(A/S)$ such that $C\cup m$ is bad for $(\tau,p)$.  Since $p$ is a bad point, there exists $p'\in (\pi\circ f)^{-1}(p)$ and a section of $\gamma$ of $X$ over $C\cup m$ mapping $b$ to $p'$ such that either $\gamma_{\tau}$ cannot be extended, the extension is not unique, or $f\circ\gamma$ cannot be extended.  If $f\circ\gamma$ cannot be extended, $\gamma$ does not have an extension.  If $\gamma_{\tau}$ cannot be extended, then $\gamma$ cannot be extended.  If the extension is not unique, these different extensions gives distinct extensions of $\gamma$ as in the proof of the first part.  The statement (3) follows from the same proof as (1).
\end{proof}

\begin{corollary}\label{verygeneralmatching}
Fix a very general point $b\in S$ and a point $p$ in $\iso(A)_b$.  Then, for a very general conic $C$ and a very general genus-0, degree-$d$ curve $m$ containing $b$ with $d$ even and $d\ge 2$ such that $C$ and $m$ intersect at a very general point, every section in $\sections_b^p(X_{C\cup m}/C\cup m)$ is the restriction of a unique section in $\sections_b^p(X/S)$.
\end{corollary}

\begin{proof}
Consider the very general subset $S^{\circ}$ as in Lemma~\ref{matching}.  For a very general $b$, and very general $C\cup m$, the restriction of sections
\[\sections_b^p(A/S)\to\sections_b^p(A_{C\cup m}/C\cup m)\]
is bijective by Corollary~\ref{inductiveII}.  

Let $c\in m\cap S^{\circ}$ be a closed point.  For every $\sigma$ in $\sections_b^p(A/S)$, there is a very general family of conic curves $\mathcal{N}_2(\sigma,c)$ such that every section of $f_{\sigma}$ over $C$ extends to a unique section of $f_{\sigma}$ by Bertini's theorem (Theorem~\ref{bertinisectionsrevised}).  Take a very general conic $C$ that is contained in $\mathcal{N}_2(\sigma,c)$ for every $\sigma\in \sections_b^p(A/S)$.  Now, for every section $\gamma\in \sections_b^p(X_{C\cup m}/C\cup m)$, $f\circ\gamma$ is contained in a unique section $\sigma\in \sections_b^p(A/S)$.  Therefore, by Lemma~\ref{matching}, $\gamma$ is contained in a unique section of $\sections_b^p(X/S)$.
\end{proof}

\begin{remark}\label{modulibyisotrivialfactor}
For a fixed $p\in \iso(A)_b$, Corollary~\ref{verygeneralmatching} claims the existence of good curve-pairs for sections in $\sections_b^p(A/S)$.  However, such a good curve-pair might be bad for other choices $p_0\in\iso(A)_b$.  And as we vary the point $p_0$, the bad sets $\mathcal{D}_b^{p_0}$ might sweep out the moduli space $\mathcal{M}_d({\mathbb{P}^n_{\mathfrak{S}}}{/\mathfrak{S}}, \tau)_{\mu(b)}$.  To resolve this problem, we have to increase the degree of curve-pairs (Theorem~\ref{generalizedjasonmaintheorem}).
\end{remark}

\subsection{Main Theorem}

Now, we can give the proof of Theorem~\ref{generalizedjasonmaintheorem}.

\begin{proof}
Let $C_1$ be a very general conic curve containing a very general point $b_1\in S$.  Let $\mathcal{M}_{2,b_1}$ be the image of $\phi_3(\mathcal{D}_{b_1})$, i.e. the set of bad points $p_1\in\iso(A)_{b_1}$ for $C_1$.  By Corollary~\ref{verygeneralmatching}, $\mathcal{M}_{2,b_1}$ is a proper subset of $\iso(A)_{b_1}$.  Take $C_2$ a very general conic intersecting with $C_1$ at a very general point $c_2$ and a very general point $b_2$ on $C_2$. Denote by $\Delta'_2(C_1\cup C_2,b_1)$ the union of the images of $C_1\cup C_2$ under bad sections $\sigma$ of $A$ over $S$ mapping $b_1$ to some $p\in\mathcal{M}_{2,b_1}$.  Let $\Delta_2(C_1\cup C_2,b_1)$ be the image of $\Delta'_2(C_1\cup C_2,b_1)$ in $\iso(A)$ under $f\circ\pi$.  Then, $\Delta_2(C_1\cup C_2,b_1)$ is contained in $\rho^{-1}(C_1\cup C_2)$ and $\Delta_2(C_1\cup C_2,b_1)\cap \mathcal{M}_{2,b_1}$ equals $\mathcal{M}_{2,b_1}$.    Define $\Delta_2(C_1\cup C_2,b_2)$ in the same way for points in $\mathcal{M}_{2,b_2}$.

By choosing $C_2$, $c_2$ and $b_2$ very generally, $\Delta_2(C_1\cup C_2,b_2)\cap\iso(A)_{b_1}$ will intersect $\mathcal{M}_{2,b_1}$ transversally.  Moreover, since $C_1\cup_{c_2}C_2$ is very general, we may assume that 
\[\sections(A/S)\to\sections(A_{C_1\cup_{c_2}C_2}/C_1\cup_{c_2}C_2)\]
are bijective by Corollary~\ref{inductiveII}.  

Let $p$ be a point in $\mathcal{M}_{2,b_1}$, but not in $\Delta_2(C_1\cup C_2,b_2)$.  Let $\sigma$ be a section in $\sections_{b_1}^p(A/S)$.  If $\sigma(b_2)$ does not belong to $\mathcal{M}_{2,b_2}$, $C_1\cup C_2$ is good for $(\sigma, b_1,p)$.  If $\sigma(b_2)$ is in $\mathcal{M}_{2,b_2}$, then $\sigma(b_1)$ is in $\Delta_2(C_1\cup C_2,b_2)$, which contradicts the choice of $p$.  Thus, $C_1\cup C_2$ is good for every section in $\sections_{b_1}^p(A/S)$, and $p$ is a good point for this marked curve-pair.  Now, take a point $p_1'$ in $(\pi\circ f)^{-1}(p)$ where $p$ is in the set $\Delta_2(C_1\cup C_2,b_2)\cap \iso(A)_{b_1}$, but not in $\mathcal{M}_{2,b_1}$.  Denote by $\gamma$ a section of $X$ over $C_1\cup C_2$ mapping $b_1$ to $p_1'$.  Let $p_2'=\gamma(b_2)$.  Denote by $p_1$, resp. $p_2$, the image of $p_1'$, resp. $p_2'$ in $\iso(A)$.  Let $\sigma$ be a section of $A$ over $S$ extending $f\circ\gamma$.  Since $p_1$ is not in $\mathcal{M}_{2,b_1}$, $C_1\cup C_2$ is good for $(\sigma, b_2,p_2)$.  By Lemma~\ref{matching}, $\gamma$ extends to a unique section of $X$ over $S$ mapping $b_2$ to $p_2'$ and $b_1$ to $p_1'$.  Thus, by the second part of Lemma~\ref{matching}, $p$ is a good point for $(b_1,C_1,c_2,C_2)$.

Denote by $\mathcal{M}_{4,b_1}$ the set of bad points of $(b_1,C_1,c_2,C_2)$.  Then, by the above argument, $\mathcal{M}_{4,b_1}$ is contained in the intersection of $\Delta_2(C_1\cup C_2,b_2)$ and $\mathcal{M}_{2,b_1}$.  Therefore, $\dim\mathcal{M}_{4,b_1}$ is strictly less than $\dim\mathcal{M}_{2,b_1}$.  Let $C_{1,2}$ be a very general, genus-0, degree-4, irreducible, smooth curve containing $b_1$.  By Corollary~\ref{inductionI} and Lemma~\ref{matching}, the bad set of points is contained in $\mathcal{M}_{4,b_1}$.  Denote the bad points for $(C_{1,2},b_1)$ by $\mathcal{M}_{4,b_1}'$.  Attach a very general conic $C_3$ to $C_{1,2}$ at a very general point $c_3$.  Then inductively, we get a decreasing sequence of dimensions
\[\dim\mathcal{M}_{2,b_1}>\dim\mathcal{M}_{4,b_1}\ge\dim\mathcal{M}_{4,b_1}'>\dim\mathcal{M}_{6,b_1}\ge\dim\mathcal{M}_{6,b_1}'>\cdots.\]
Then, for $d>2e$ an even number, the bad set for $(b_1,m,c,C)$ of degree-$(d+2)$ is empty, and hence every section of $X$ over $C\cup m$ is the restriction of a unique section.  And, by Corollary~\ref{inductionI}, this is also true for very general irreducible smooth curve of degree-$(d+2)$.
\end{proof}

\begin{appendices}

\section{The Bi-gon Lemma}\label{thebigonlemma}

Let $k$ be an algebraically closed field.  In the statement of the
Bi-gon Lemma, also $k$ will be uncountable.
Let $\bB$ be an irreducible,
quasi-projective $k$-scheme of dimension $\geq 2$.

\begin{definition} \label{defn-defective} 
For every $k$-morphism of locally finite type $k$-schemes, $f:R\to S$,
for every integer $\delta \geq 0$, the $\delta$-\textbf{locus of} $f$, 
$E_{f,\geq \delta}\subseteq R$, is the union of all
irreducible components of fibers of $f$ that have dimension $\geq
\delta$.  The $\delta$-\textbf{image of} $f$, $F_{f,\geq
  \delta}\subseteq S$, is the image under $f$ of $E_{f,\geq \delta}$.
\end{definition}

\begin{lemma}(\cite{Hart}, Exercise II.3.22, p.95)
\label{lem-defective} 
For every locally finite type morphism $f$ and every integer $\delta
\geq 0$, 
the subset $E_{f,\geq \delta}$ of $R$ is closed.  If also $f$ is
quasi-compact, resp. proper, 
then the subset $F_{f,\geq \delta}$ of $S$ is constructible, resp. closed.
\end{lemma}

\begin{definition} \label{defn-omega} 
For every proper, surjective morphism $\rho:Y\to \bB$, 
for every pair $(T,w)$ of an integral scheme $T$ and a dominant,
finite type morphism, 
$$
w:T\to \bB,
$$
a $(T,w)$-\textbf{multisection of}
$\rho$ is a pair $(\Omega,v)$ of an irreducible scheme $\Omega$ and a
proper morphism $v=(v_Y,v_T)$,
$$
v:\Omega \to Y\times_{\bB}T, \ \ v_Y:\Omega \to Y, \ \ v_T:\Omega
\to T,
$$ 
such that $v_T$ is surjective and generically finite.  Since $v$ is
proper, also the image $(v(\Omega),v(\Omega)\hookrightarrow
Y\times_{\bB}T)$ is a $(T,w)$-multisection of $\rho$.  This is the
\textbf{image multisection} of $(\Omega,v)$.

For every pair $((\Omega',v'),(\Omega'',v''))$ 
of $(T,w)$-multisections, denote
the fiber product of $v'$ and $v''$ by 
$$
(\pi':P_{v',v''} \to \Omega',\ \
\pi'':P_{v',v''} \to \Omega''), \ \ v'\circ \pi' = v''\circ \pi''.
$$
The
\textbf{special subset} $S_{v',v''}$
\textbf{of the pair}
is the closed image in $T$ 
of $P_{v',v''}$.
\end{definition}

\begin{definition} \label{defn-witness} 
\textbf{(i).} For an integral, quasi-projective 
$k$-scheme $M$, a \textbf{family of smooth, proper, connected curves
  over} $M$ is a smooth, proper morphism,
$$
\ol{u}_M:\ol{\Cc} \to M,
$$
whose geometric fibers are connected curves.

\mni
\textbf{(ii).}
For every open immersion 
$$
\iota:\Cc \to \ol{\Cc}
$$
whose image is dense in every fiber of $\ol{u}$, the composite
morphism $u_M = \ol{u}_M\circ \iota$ is 
a \textbf{family of smoothly compactifiable curves over} $M$.

\mni
\textbf{(iii).}
A \textbf{family of curves to $\bB$} is a pair $(M,u)$
of an irreducible, quasi-projective $k$-scheme $M$ and a proper
morphism $u=(u_Y,u_M)$,
$$
u:\Cc \to \bB\times_{\SP k}M, \ \ u_{\bB}:\Cc \to \bB, \ \ u_M:\Cc
\to M,
$$
such that $u_M$ is a family of smoothly compactifiable curves over $M$.

\mni
\textbf{(iv).}
The family of curves to $\bB$ is
\textbf{connecting}, resp. \textbf{minimally connecting}, 
if the following induced $k$-morphism is dominant, resp. dominant and
generically finite, 
$$
u^{(2)}:\Cc\times_M \Cc \to \bB \times_{\SP k} \bB, 
\ \ 
\text{pr}_{i}\circ
u^{(2)} = u_{\bB}\circ \text{pr}_i, i=1,2.
$$
By definition, both $u_M\circ \text{pr}_1$ and $u_M\circ \text{pr}_2$
are equal as morphisms from $\Cc\times_M \Cc$ to $M$; denote this
common morphism by $\wt{u}_M$.  
Denote by $\wt{u}^{(2)}$ the induced morphism
$$
(u^{(2)},\wt{u}_M):\Cc\times_M \Cc \to \bB\times_{\SP k} \bB \times_{\SP k} M.
$$
\textbf{(v).}
A connecting family of curves to $\bB$ is a
\textbf{Bertini family} if for every 
integral $k$-scheme $\wt{B}$ and for
every 
finite, surjective $k$-morphism
$\phi:\wt{\bB}\to \bB$, the induced morphism 
$\wt{\bB}\times_{\bB} \Cc \to M$ has
integral geometric generic fiber.  Denote by $M_\phi$ the maximal open
subscheme of $M$ over which $\wt{\bB}\times_{\bB} \Cc$ has integral
geometric fibers.

\mni
\textbf{(vi).}
An integral closed subvariety $N$ of $M$ is
\textbf{transversal}
if the following family  of curves to $\bB$ is minimally connecting,
$$
(N,u\times \text{Id}_N:\Cc\times_M \to \bB \times_{\SP k} M \times_M N).
$$
Such a subvariety is $\phi$-\textbf{Bertini} if $N$ intersects the
open $M_\phi$.

\mni
\textbf{(vii).}
The \textbf{transversal dimension} is $d:= 2\text{dim}(\bB)-2$.  For
every integer $e$ with $0\leq e\leq d$, 
the
\textbf{transversal Grassmannian} $G_e$ is the open subscheme
of the Grassmannian
parameterizing $e$-dimensional linear sections $N$ of $M$.
\end{definition}

\begin{lemma} \label{lem-transversal-a} 
For every connecting family of curves $(M,u)$, for every integer $e$
with $0<e\leq d$, a general point of $G_e$ parameterizes a linear
section $N$ of $M$ that is geometrically integral.
\end{lemma}

\begin{proof}
This follows from a Bertini Connectedness Theorem, \cite{jouan} Th\'{e}or\`{e}me 6.10.
\end{proof}

\begin{lemma} \label{lem-transversal} 
For every connecting, Bertini family $(M,u)$,
for every integer $e$ with $0\leq e \leq d$, a general point of $G_e$
parameterizes a linear section $N$ of $M$ such that the induced
morphism $u^{(2)}_N$ is generically finite.  
For every finite, surjective 
$k$-morphism $\phi:\wt{\bB}\to \bB$, every general $N\in G_d$
is transversal and $\phi$-Bertini.
\end{lemma}

\begin{proof}
Generic finiteness of $u^{(2)}_N$ is proved by induction on $e$.  The
base case is when $e=0$.  
Since the family is connecting, the morphism $u$ is generically finite
to its image.  Thus, the morphism $\wt{u}^{(2)}$ is generically finite to
its image.  The $1$-relative locus $E$ of $\wt{u}^{(2)}$ is a proper,
closed subset of $\Cc\times_M \Cc$.  For the induced morphism,
$$
\wt{u}_M|_E:E\to M,
$$
the $2$-relative locus $E_{\geq 2}$ of this morphism is a closed subset of
$E$.
Since $E$ is a proper closed subset of $\Cc\times_M \Cc$, and since
the geometric generic fiber of $\wt{u}_M$ is irreducible, the proper
closed subset $E_{\geq 2}$ is disjoint from this fiber.
Thus, the
image $F_{\geq 2}$ of $E_{\geq 2}$ in $M$ is a constructible subset that
does not contain the generic point, i.e., it is not
Zariski dense.  Denote by $M^o$ the open subset of $M$ that is the
complement of the closure of the image of $F_{\geq 2}$.  
For every singleton $N$ of a closed point of $M^o$, 
the restrictions to $N$ of $\wt{u}^{(2)}$ and $u^{(2)}$ are equal;
refer to this common restriction by $u^{(2)}_N$.  Since $\wt{u}^{(2)}$
is generically finite on the fiber over $N$ by construction, also
$u^{(2)}_N$ is generically finite.  This establishes the base case.

For the induction step, assume that the result is proved for an integer
$e$ satisfying 
$0\leq e< 2\text{dim}(B)-2$.  Then for a general linear subvariety $N$ of
$M$ of dimension $e$, the image of $u_N^{(2)}$ has dimension $e+2 <
2\text{dim}(B)$.  Thus, the image is not Zariski dense.  Since
$u^{(2)}$ is dominant, a general point of $\bB\times_{\SP k}\bB$ is
contained in the image of $u^{(2)}$ over a general point of $M$, say
$m$.  Let $N'$ be the intersection of $M$ with the span of $N$ and
$m$.  Then $N'$ is a linear subvariety of $M$ of dimension $e+1$.  By
Lemma \ref{lem-transversal-a}, for $N$ general and for $m$ general,
the linear section $N'$ is geometrically integral.  
Thus, the image of $u^{(2)}_{N'}$ is a geometrically integral scheme
that is strictly larger than the image of
$u_N^{(2)}$.  Thus, the dimension of the image of $u^{(2)}_{N'}$
is strictly larger than the dimension of the image of $u^{(2)}_N$.
Since $u^{(2)}_N$ is generically finite, and since $N'$ has dimension
precisely $1$ larger than the dimension of $N$, also $u^{(2)}_{N'}$
has dimension precisely $1$ larger than the dimension of $u^{(2)}_N$.  
Thus,
also $u^{(2)}_{N'}$ is generically finite to its image.  

In particular, for
$e$ equal to $d$, since $u^{(2)}_N$ is generically
finite and the domain and target both have the same dimension, the
image of $u^{(2)}_N$ contains a nonempty Zariski open subset of
$\bB\times_{\SP k}\bB$.  By hypothesis, $\bB\times_{\SP k} \bB$ is
integral.  Thus, the image contains a dense Zariski open subset of
$\bB\times_{\SP k} \bB$.  Therefore $(N,u_N)$ is a minimally
connecting family, i.e., $N$ is transversal.

Finally, by hypothesis, the open subscheme $M_{\phi}$ contains the
generic point of $M$, and hence this open subscheme is dense.
Therefore, a general $N$ intersects $M_{\phi}$.
\end{proof}

\begin{lemma} \label{lem-transversal4} 
For every minimally connecting family of curves to $\bB$,
$$
(N,(u_{N,\bB},u_N):\Cc_N \to
\bB\times_{\SP k} N),
$$ 
for every $(\Cc_N,u_{N,\bB})$-multisection $(\Omega,v)$ of $\rho$,
the scheme
$\Omega \times_N \Cc_N$ is irreducible, and the following composition
$\wt{v}_B$ is a multisection of $\rho$ relative to
$\bB\times_{\SP k} \bB \xrightarrow{\text{pr}_1} \bB$,
$$
\Omega \times_N \Cc_N \xrightarrow{v_N\times \text{Id}_{\Cc_N}}
Y\times_{\bB} \Cc_N \times_N \Cc_N \xrightarrow{\text{Id}_Y \times
  u_{N,\bB}^{[2]}} Y\times_{\bB} \bB \times_{\SP k} \bB.
$$
\end{lemma}

\begin{proof}
Since the morphism
$$
u_N:\Cc_N\to N
$$
is flat with integral geometric fibers, the following base change
morphism is flat with integral geometric fibers,
$$
\text{pr}_{\Omega}: \Omega \times_N \Cc_N \to \Omega.
$$
Since $\Omega$ is irreducible, and since $\text{pr}_{\Omega}$
is flat with integral geometric generic fiber, also $\Omega \times_N
\Cc_N$ is irreducible.  

Since $v_{\Cc}$ is surjective and generically finite, and since $u_N$
is flat, also the following morphism is surjective and generically
finite,
$$
v_{\Cc}\times \text{Id}_{\Cc_N}:\Omega \times_N \Cc_N \to \Cc_N
\times_N \Cc_N.
$$
Since $(N,u_N)$ is minimally connecting, the following morphism
is dominant and generically finite,
$$
u_N^{[2]}:\Cc_N\times_N \Cc_N \to \bB \times_{\SP k} \bB.
$$
Thus, the composition is dominant and generically finite. This
composition equals the composition of $\wt{v}_B$ with the morphism
$$
\rho\times \text{Id}_{\bB} :Y\times_{\SP k} \bB \to \bB\times_{\SP k} \bB.
$$
Thus, the morphism $\wt{v}_B$ is a multisection of $\rho$.
\end{proof}

For every pair of connecting families of curves to $\bB$, 
$$
(M',u':\Cc'\to \bB\times_{\SP k}M'), \ \
(M'',u'':\Cc''\to \bB\times_{\SP k}M''),
$$
denote by $G$, resp. by $G'$, the open subscheme of the
Grassmannian parameterizing $d$-dimensional linear sections $N'$ of
$M'$, resp. $N''$ of $M''$, that are transversal; by Lemma
\ref{lem-transversal}, there is a dense open subscheme parameterizing
linear sections that are transversal.

\begin{lemma} \label{lem-transversal5} 
For every pair of connecting families of curves to $\bB$ as above that
are Bertini families, 
for every pair 
$$
(\Omega',v'), \ \ 
(\Omega'',v'')
$$
of a $(\Cc',u'_{\bB})$-multisection of
$\rho$ and a $(\Cc'',u''_{\bB})$-multisection of $\rho$, for a general
pair $(N',N'') \in G'\times_{\SP k} G''$, 
the families
$(N',u'\times \text{Id}_{N'})$ and $(N'',u'' \times \text{Id}_{N''})$ 
are minimal connecting families
of curves to $\bB$.  Also, for a general pair $(b',b'')\in \bB\times_{\SP
  k} \bB$,
the family $(N',u'_{N'})$, resp. $(N'',u''_{N''})$,
is a Bertini family for the image in $Y$ of the
multisection $\wt{v}'_{\bB,b'}$, resp. $\wt{v}''_{\bB,b''}$ of $\rho$,
obtained by restricting to the fiber of $\text{pr}_2:B\times_{\SP
  k}B\to B$ over $b'$, resp. over $b''$. 
\end{lemma}

\begin{proof}
By Lemma \ref{lem-transversal4}, each of
$(\Omega'_{N'}\times_{N'}\Cc'_{N'},\wt{v}'_{\bB})$ and
$(\Omega''_{N''}\times_{N''}\Cc''_{N''},\wt{v}''_{\bB})$ is a
$\text{pr}_1$-multisection of $\rho$.  Thus, for general $(b',b'')\in
\bB\times_{\SP k} \bB$,
the restriction of the $\text{pr}_1$-multisection
$\Omega'_{N'}\times_{N'} \Cc'_{N'}$, resp. $\Omega''_{N''}\times_{N''}
\Cc''_{N''}$, to the $\text{pr}_2$-fiber over 
$b'$, resp. over $b''$, maps dominantly and generically finitely to
$B$, i.e., each of the finitely many
irreducible components of the the restriction is a
$u'$-multisection of $\rho$, resp. a $u''$-multisection of $\rho$.
Denote the image in $Y$ of this finite union of
multisection by $\wt{v}'_{\bB,b'}$, resp. by $\wt{v}''_{\bB,b''}$.  
By Lemma \ref{lem-transversal}, for $N''$ general applied to the
finitely many irreducible components of $\wt{v}'_{\bB,b'}$, 
the
family $(N'',u''_{N''})$ is transversal and Bertini relative to
$\wt{v}'_{\bB,b'}$. 
Similarly, for $N'$ general, the family
$(N',u'_{N'})$ is transversal and Bertini relative to
$\wt{v}''_{\bB,b''}$.  
\end{proof}

\begin{lemma} \label{lem-transversal6} 
Assume that $k$ is algebraically closed and uncountable.
With the same hypotheses as above, for a countable family of
$(M',u')$-multisections, $(\Omega'_{i'},v'_{i'})_{i'\in I'}$, with
pairwise distinct images in $Y\times_{\bB}\Cc'$,
resp. for a countable family of $(M'',u'')$-multisections,
$(\Omega''_{i''},v''_{i''})_{i''\in I''}$, with pairwise distinct
images in $Y\times_{\bB} \Cc''$,
if $(N',N'')\in G'\times_{\SP k} G''$ and $(b',b'')\in \bB\times_{\SP
  k} \bB$ are very general, then for every 
$(i',i'')\in I'\times I''$, the conclusion holds for
$(\Omega'_{i'},v'_{i'})$ and $(\Omega''_{i''},v''_{i''})$.  
\end{lemma}

\begin{proof}
For each $(i',i'')$, by Lemma \ref{lem-transversal5}, there exists a
dense open $W_{i',i''}$ of $G'\times_{\SP k} G''\times_{\SP k} \bB
\times_{\SP k} \bB$ parameterizing $(N',N'',b',b'')$ such
that Lemma \ref{lem-transversal5} holds.  Thus, for every
$(N',N'',b',b'')$ in the countable intersection $\cap_{(i',i'')}
W_{i',i''}$, the conclusion of the lemma holds for every
$(\Omega'_{i'},v'_{i'})$ and $(\Omega''_{i''},v''_{i''})$.  
\end{proof}

\begin{lemma}(The Bi-gon Lemma) \label{lem-transversal3}
With hypotheses as in the previous lemma, for a very general
$(N',N'',b',b,b'')$ in $G'\times_{\SP k} G'' \times_{\SP k} \bB
\times_{\SP k} \bB\times_{\SP k} \bB$, 
for a very
general $2$-pointed bi-gon $(\Cc=\Cc'_{t'} \cup_b \Cc''_{t''},b',b'')$ 
parameterized by a point $t'\in N'$, resp.
$t''\in N''$, whose curve $\Cc'_{t'}$ contains $b$ and $b'$,
resp. whose curve $\Cc''_{t''}$ contains $b$ and
$b''$, 
the only sections $\sigma$ of $\rho$ over $\Cc$ whose restriction to
$\Cc'_{t'}$ is 
in some $\Omega'_{M',N',b',i'}$ and whose restriction to $\Cc''_{t''}$
is in some $\Omega''_{M'',N'',b'',i''}$ are those that come from global
sections $\Omega'_{M',N',b',i'} = \Omega = \Omega''_{M'',N'',b'',i''}$
over $B$.
\end{lemma}

\begin{proof}
Let $W$ denote the countable intersection of $W_{i',i''}$ inside
$G'\times_{\SP k} G'' \times_{\SP k} \bB \times_{\SP k} \bB$ as in the
proof of the previous lemma.  Let $(N',N'',b',b'')$ be an element of
$W$.  
Consider the
countable collection of image multisections $(\Omega'_{M',N',b',i'})_{i'}$ and
$(\Omega''_{M'',N'',b'',i''})$ of $\rho$ as closed subschemes of $Y$.
For every pair $(i_1',i_2')$ of distinct elements of $I'$, 
the special subset $S_{i_1',i_2'}$ associated to
$\Omega'_{M',N',b',i_1'}$ and $\Omega'_{M',N',b',i_2'})$ 
is a proper
closed subset of $B$, and similarly for the special subset
$S_{i_1'',i_2''}$ associated to every pair $(i_1'',i_2'')$ of distinct
elements of $I''$.
Finally, for every $i'\in I'$ and every $i''\in I''$,
the special subset
$S_{i',i''}$ associated to 
$\Omega'_{M',N',b',i'}$ and $\Omega''_{M'',N'',b'',i''})$ is a proper
closed subset except in those cases where $\Omega'_{M',N',b',i'}$
equals $\Omega''_{M'',N'',b'',i''}$.  

Choose $b$ to be a very general
point of $\bB$ that is contained in none of these special subsets that
is a proper closed
subsets of $\bB$.   The matching condition at $b$ for a section $\sigma$
implies that $\sigma(\Cc)$ is contained in 
$\Omega'_{M',N',b',i'} = \Omega''_{M'',N'',b'',i''}$ for
unique $\Omega'_{M',N',b',i'}$ and $\Omega''_{M'',N'',b'',i''}$ in their
respective countable multisections.  

By the previous lemma,
for every
$i'\in I'$, 
if the restriction of $\Omega_{M',N',b',i'}$ over $\Cc''_{t''}$
has a section, then the multisection $\Omega_{M',N',b',i'}$ is a
global section.  Similarly, for every $i''\in I''$, if the restriction
of $\Omega_{M'',N'',b'',i''}$ over $\Cc'_{t'}$ has a section, then the
multisection $\Omega_{M'',N'',b'',i''}$ is a global section.
Thus, for every section $\sigma$ as in the previous paragraph,
$\Omega'_{M',N',b',i'} = \Omega''_{M'',N'',b'',i''}$ is a global section.
\end{proof}

\section{Proofs of Two Lemmas}\label{Proofoftwolemmas}

For completeness, we prove two lemmas that are already in the literature in this appendix.

\mni
Proof of Lemma~\ref{cartiersmooth}.

\begin{proof}
This is a local problem, so we can assume that $S=\spec R$ and $X$ is a closed subscheme of $W=\mathbb{A}^n_R$ defined by $g_1, \cdots, g_r$.  Let $Z$ is defined by $g$ in $\mathcal{O}_X$.  Let $z\in Z$ and $dw_1$,$\cdots$, $dw_n$ be a basis of $(\Omega_{W/S}^1)_z$.  Then, up to a re-indexing, $g_{t+1}$, $\cdots$, $g_{n-t-2}$, $g$ generate the ideal sheaf defining $Z$ and $dw_1$,$\cdots$, $dw_t$, $dg_{t+1}$, $\cdots$, $dg_{n-t-2}$, $dg$  generate $(\Omega_{W/S}^1)_z$ (\cite{BLR}, Prop.2.2/7, p.39).  Let $Y$ be the closed subscheme of $\mathbb{A}^n_R$ defined by $g_{t+1}$, $\cdots$, $g_{n-t-2}$.  Then, we have $X\subset Y$, a closed subscheme.  Since both $dw_1$,$\cdots$, $dw_n$ and $dw_1$,$\cdots$, $dw_t$, $dg_{t+1}$, $\cdots$, $dg_{n-t-2}$, $dg$ are basis of $(\Omega_{W/S}^1)_z$, there exists some $w_{t+1}$ such that $dw_1$,$\cdots$, $dw_t$, $dw_{t+1}$, $dg_{t+1}$, $\cdots$, $dg_{n-t-2}$ form a basis of $(\Omega_{W/S}^1)_z$.  By the Jacobian criterion, $Y$ is smooth at $z$ over $S$.  Thus, locally at $z$, $Y$ is regular (\cite{Liuqing}, Theorem 4.3.36, p.142).  Therefore, $X$ and $Y$ are regular schemes of the same dimension locally at $z$ with $X\subset Y$.  We get $X=Y$ locally at $z$, and hence $X$ is smooth at $z$ over $S$.  Consider every point $z\in Z$, there will be an open subset $U\subset X$ such that $U\rightarrow S$ is smooth.  
\end{proof}

\mni
Proof of Lemma~\ref{abelianvarieties}.

\begin{proof}
Since Abelian varieties are projective, there exists a very ample sheaf $\mathcal{L}$ on $A\times_k B$.  Then, for every homomorphism $u: (A,0)\rightarrow (B,0)$, the graph $G_u$ in $A\times_k B$ has a Hilbert polynomial $P(t)$ with respect to $\mathcal{L}$.  Let $\homo^{P}_k(A,B)$ be the scheme parameterizing homomorphisms from $A$ to $B$ with Hilbert polynomial $P(t)$.  Then, $\homo^{P}_k(A,B)$ is quasi-projective over $k$.  Now, take a homomorphism $u$ from $A$ to $B$.  The Zariski tangent space of $\homo^{P}_k(A,B)$ at $[u]$ is isomorphism to the $k$-vector space of global sections of
\[\mathcal{E}=\homo_{\mathcal{O}_A}(u^*\Omega^1_{B/k},\mathcal{I}_0)\]
where $\mathcal{I}_0$ is the ideal sheaf defining the origin $0$ in $A$.  Since $B$ is an Abelian variety, $\Omega^1_{B/k}$ is isomorphic to the trivial locally free sheaf $\Omega_0\otimes_k\mathcal{O}_B$ where $\Omega_0$ is dual space $T_{B,0}^*$ of the Zariski tangent space $T_{B,0}$ of $B$ at the origin (\cite{mfdabelian}, (iii), p.39).  Thus, $\mathcal{E}$ is equal to

\[\homo_k(\Omega_0, k)\otimes_k\mathcal{I}_0.\]
Since $A$ is projective and by the exact sequence of the ideal sheaf $\mathcal{I}_0$ and structure sheaf of the origin $\mathcal{O}_A/\mathcal{I}_0$ (\cite{Liuqing}, Cor.3.3.21), there is no nonzero global section for $\mathcal{I}_0$, and hence the finite direct sum $\mathcal{E}$ of $\mathcal{I}_0$ does not have nonzero global section.  Therefore, $\hg^0(A,\mathcal{E})=0$ and $[u]$ is an isolated point of the quasi-projective $k$-variety $\homo^{P}_k(A,B)$.  So there are at most countably many such homomorphisms.
\end{proof}

\section{Table of Notations}

\begin{table}[H]
\centering
\begin{tabular}{r c p{10cm} }
\toprule
$W$ &  & Notation~\ref{projectivecompactification}\\
$A_0$ &  & Notation~\ref{chowtrace}\\
$Q$ &  & strongly nonisotrivial factor in Notation~\ref{chowtrace}\\
$\rho_{iso}$ &  & isogeny in Notation~\ref{chowtrace}\\  
$\iso(A)$ &  & Definition-Lemma~\ref{isotrivialfactor}\\
$\pi:A\to\iso(A)$ &  & Definition-Lemma~\ref{isotrivialquotient}\\
$\sections_b^p(\cdot/ \cdot)$ &  & Notation~\ref{sectionsdefinition}\\
$\mathfrak{S}$ &  & Notation~\ref{grothendieckpifunctor}\\  
$\mathcal{M}_{d+2}(\mathbb{P}^n_k, \tau)$ & & Notation~\ref{md2}\\
$\rho_{ev}$ & & evaluation map of $\mathcal{M}_{d+2}(\mathbb{P}^n_k, \tau)$\\
$\mathcal{M}_{d+2}({\mathbb{P}^n_{\mathfrak{S}}}{/\mathfrak{S}}, \tau)$ & & Notation~\ref{mds}\\
$\rho_{ev}^{\mathfrak{S}}$ & & evaluation map of $\mathcal{M}_{d+2}({\mathbb{P}^n_{\mathfrak{S}}}{/\mathfrak{S}}, \tau)$\\
$\mathcal{M}_{d+2}({\mathbb{P}^n_{\mathfrak{S}}}{/\mathfrak{S}}, \tau)_{\mu(b)}$ & & Notation~\ref{mdbb}\\
$\mathcal{M}_{0,d}({\mathbb{P}^n_{\mathfrak{S}}}{/\mathfrak{S}}, 1)$ & & Notation~\ref{md01}\\
$\mathcal{M}_{d+2}(\mathbb{P}^n_k, \tau, b)$ & & Notation~\ref{md0b}\\
$\mathcal{M}_{0,d}(\mathbb{P}^n_k, b)$ & & Notation~\ref{md0bb}\\
$\mathcal{W}$ & & Notation~\ref{www}, an open dense subset of $\mathcal{M}_{d+2}({\mathbb{P}^n_{\mathfrak{S}}}{/\mathfrak{S}}, \tau)$\\
$\mathcal{D}_b$, $\mathcal{D}_b^p$ & & Definition~\ref{badsets}\\
$\mathcal{H}=\mathcal{M}_{d+2}(\mathbb{P}^n_k, \varepsilon, b)$ & & Notation~\ref{mdvb}\\
$\mathcal{X}$ & & Notation~\ref{xxx}\\
$\mathcal{H}_0$, $\mathcal{H}_1$, $\mathcal{X}_0$, $\mathcal{X}_1$ & & Notation~\ref{xxx0}\\
$\mathcal{C}_{\mathcal{H}}$, $\mathcal{C}_{\mathcal{H}_0}$, $\mathcal{C}_{\mathcal{H}_1}$ & & Notation~\ref{ch01}\\
$H$ & & Notation~\ref{universalsectionlifting}\\
$\varrho$ & & Notation~\ref{universalhhhhhhhh}\\
$H_0$ & & Theorem~\ref{keylemma}\\
\bottomrule
\end{tabular}
\label{tab:TableOfNotationForMyResearch}
\end{table}

\end{appendices}

\vspace{1em}

\noindent\small{\textsc{Mathematics Department, Stony Brook University, Stony Brook, NY, 11794} }

\noindent\small{Email: \texttt{santai.qu@stonybrook.edu}}

 \end{document}